\documentclass[english, 10pt]{amsart}
\usepackage{amsxtra,amscd,amsthm,amsmath,amssymb}
\usepackage{longtable}
\usepackage{bigstrut}
\usepackage{enumerate}
\usepackage{verbatim}
\usepackage{pdfsync}

\usepackage{graphicx}
\usepackage{xspace}
\usepackage{ifthen}
\usepackage{tabularx}
\usepackage[all,cmtip,line]{xy}

\usepackage{color}
\usepackage{mathrsfs}

\renewcommand{\l}{\ell}

\usepackage{amsthm}

\theoremstyle{plain}
\numberwithin{equation}{subsection}
\newtheorem{thm}[equation]{Theorem}
\newtheorem{prop}[equation]{Proposition}
\newtheorem{lemma}[equation]{Lemma}
\newtheorem{Definition}[equation]{Definition}

\newtheorem{Remark}[equation]{Remark}

\newtheorem{cor}[equation]{Corollary}

\theoremstyle{remark}
\newtheorem{para}[equation]{\bf}
\theoremstyle{plain}
\renewcommand{\subsubsection}{\addtocounter{equation}{1}{\vskip 6pt \noindent\  \bf\arabic{section}.\arabic{subsection}.\arabic{equation}}\,}
\theoremstyle{definition}

\newcommand{\quash}[1]{}  
\newcommand{\nc}{\newcommand}
\nc{\on}{\operatorname}

\newcommand{\cal}{\mathcal}

\newcommand{\E}{{\mathscr E}}

\newcommand{\GG}{{\mathscr G}}

\newcommand{\Q}{{\mathbb Q}}

\newcommand{\Gg}{{\mathcal G}}

\newcommand{\Z}{{\mathbb Z}}
\newcommand{\F}{{\mathscr F}}
\newcommand{\ti}{\widetilde}
\newcommand{\Spec}{{\rm Spec } }
 \renewcommand{\O}{{\mathcal O}}

\newcommand\al{\vartheta^{\ell}}

\newcommand\ct{{\rm ct}}

\newcommand{\Gr}{{\rm G}}

\newcommand{\und}{\underline}

\newcommand{\pfr}{{\mathfrak p}}

\newcommand{\M}{{\mathcal M}}

\newcommand{\Kr}{{\rm K}}

\newcommand{\UU}{{\mathcal U}}

\newcommand{\Pic}{{\rm Pic}}

\newcommand{\CNT}{{\rm CNT}}

\newcommand{\Zl}{{\Z'}}

\newcommand{\DD}{{\mathcal D}}

\newcommand{\ov}{\overline}

\newcommand{\scrR}{{\mathscr R}}

\newcommand{\cl}{{\rm cl}}

\def\thfill{\null\nobreak\hfill}

\def\endproof{\thfill\vbox{\hrule
  \hbox{\vrule\hbox to 5pt{\vbox to 5pt{\vfil}\hfil}\vrule}\hrule}}

 \renewcommand{\mod}{ { \rm\,  mod\,} }

\newcommand{\Cl}{{\mathrm {Cl}}}

\newcommand{\Gal}{{\rm Gal}}

\begin{document}

 \title[Adams operations and Galois structure]{Adams operations and   Galois structure
 }
\author[G. Pappas]{G. Pappas}
\thanks{{\it Key words:}  Galois cover, Galois module, Riemann-Roch theorem, Euler characteristic}
\thanks{{\it MSC 2010: } Primary: 14L30, 14C40, 11S23, Secondary: 20C10, 14C35, 19E08, 11R33, 14F05}
 \thanks{Partially supported by NSF
grant  DMS11-02208}
\address{ Dept. of Mathematics\\ Michigan State University\\ E. Lansing\\ MI 48824-1027\\ USA}
  \date{\today}

 \begin{abstract}
 We present a new method for determining the Galois module structure 
 of the cohomology of coherent sheaves on varieties over the integers
 with a tame action of a finite group. This uses  
a novel Adams-Riemann-Roch type theorem obtained by combining the 
 K\"unneth formula with localization in equivariant $\rm K$-theory
 and classical results about cyclotomic fields. 
As an application, we show two conjectures of 
 \cite{CPTAnnals}, in the case of curves.

 \end{abstract}
 \maketitle

 \tableofcontents

\section{Introduction}

In this paper, we consider $G$-covers $\pi: X\to Y$ 
of a projective flat scheme $Y\to \Spec(\Z)$ where $G$ is a finite group. Let $\F$ be a 
$G$-equivariant coherent sheaf of
$\O_X$-modules on $X$, \emph{i.e.} a coherent   $\O_X$-module equipped  with a $G$-action compatible with the $G$-action
on the scheme $X$. 
Our main objects of study are the Zariski cohomology groups ${\rm H}^i(X,\F)$.
These are finitely generated modules
for the integral group ring $\Z[G]$. 

Let us consider the
total cohomology ${\rm R}\Gamma(X,\F)$ in the derived category
 of complexes of $\Z[G]$-modules which are bounded below. If $\pi$ is  tamely ramified,
then  the complex ${\rm R}\Gamma(X,\F)$ is ``perfect", \emph{i.e.}  isomorphic
to a bounded complex $(P^\bullet)$ of finitely generated {\sl projective} $\Z[G]$-modules
 (\cite{ChinburgTameAnnals}, \cite{ChinburgErez}, \cite{PMathAnn}).
This observation goes back to E. Noether when $X$ is the spectrum of the ring of integers 
of a number field;
in this general set-up it is due to T.~Chinburg.

\begin{Definition} {\sl We say that the cohomology of $\F$ has a normal integral basis
if there exists a bounded complex $(F^\bullet)$ of finitely generated
{\rm free} $\Z[G]$-modules which is  isomorphic to 
 ${\rm R}\Gamma(X,\F)$.}
 \end{Definition}
 
To measure the obstruction to the existence
of a normal integral basis we use the Grothendieck group $\Kr_0(\Z[G])$
of finitely generated projective $\Z[G]$-modules. We consider the projective class
group $\Cl(\Z[G])$ which is defined as the quotient of $\Kr_0(\Z[G])$
by the subgroup generated by the class of the free module $\Z[G]$.
The obstruction to the existence
of a normal integral basis for the cohomology of $\F$ is given by the (stable)
projective Euler characteristic 
$$
\overline\chi(X, \F)=\sum\nolimits_{i}(-1)^i[P^i]\ \in\ \Cl(\Z[G]) 
$$
which is independent of the choice of $(P^\bullet)$.

The main problem in the  theory of geometric Galois structure is to
understand such Euler characteristics. Often there are interesting connections
with other invariants of $X$. For example, it was shown in \cite{CEPTAnnals}
that,
under some additional hypotheses, the projective Euler characteristic
of a version of the deRham complex of $X$
can be calculated using $\epsilon$-factors of Hasse-Weil-Artin $\rm L$-functions
for the cover $X\to X/G$. Also, when $X$ is a curve over $\Z$,
the obstruction $\overline\chi(X, \O_X)$ is related, via a
suitable equivariant version of the Birch and Swinnerton-Dyer conjecture,
to the $G$-module structure 
of the Mordell-Weil and Tate-Shafarevich groups of the Jacobian of 
the generic fiber $X_\Q$
(\cite{CPTAnnals}).

We will first discuss the case when $\pi: X\to Y$ is unramified.
Let us denote by $d$ the relative dimension of $Y\to \Spec(\Z)$.

When $d=0$, the  problem
 of the existence of a normal integral basis reduces to a classical 
 question: Suppose that $N/K$ is an unramified Galois extension 
 of number fields with Galois group $G$ and consider the ring of integers $\O_N$ 
which is then a projective $\Z[G]$-module.
Take $X=\Spec(\O_N)$, $Y=\Spec(\O_K)$, $\F=\O_X$;
 then the Euler characteristic $\overline\chi(X,\O_X)$ is  the class 
$[\O_N]$ in $\Cl(\Z[G])$.  We are then  asking if $\O_N$ is ``stably free", \emph{i.e.}  if there are integers $n$ and $m$ such that
 $\O_N\oplus \Z[G]^n\simeq \Z[G]^m$. Results of A.~Fr\"ohlich and M.~Taylor 
imply that $[\O_N]$ is always $2$-torsion in $\Cl(\Z[G])$
and is trivial when the group $G$ has no irreducible
symplectic representations; this result also holds when, more generally,  $N/K$ is tamely ramified.  
Indeed, if $N/K$ is at most tamely ramified, Fr\"ohlich's conjecture (shown  by M.~Taylor \cite{MJTFrohlichConj})
explains how  to determine the class $[\O_N]$ from
the root numbers of Artin $\rm L$-functions for  irreducible symplectic representations of ${\rm Gal}(N/K)$.
In particular, ${\rm gcd}(2,\#G)\cdot [\O_N]=0$.
When $G$ is of odd order, $[\O_N]=0$ and  $\O_N$ is stably free;
when $G$ is of odd order, it then 
 follows that $\O_N$ is actually a free $\Z[G]$-module.
In general, the class of a projective module in the class group $\Cl(\Z[G])$ contains a lot of information about the isomorphism class of the module. Hence, the projective Euler characteristics that we consider in this paper 
also contain a lot of information about the Galois modules given by cohomology.  
This is not necessarily the case for the ``naive" Euler characteristics that can be easily defined
as classes in the weaker Grothendieck group of all finitely generated $G$-modules.

When $d>0$, 
some progress towards calculating  $\overline\chi(X, \F)$ for general $\F$ was achieved
after the introduction of the technique of cubic structures (\cite{PaCubeInvent}, \cite{PMathAnn}, \cite{CPTAnnals};
the paper \cite{CEPTAnnals} only dealt with  the deRham complex).
This technique is very effective when {\sl all} the Sylow subgroups of $G$ are abelian. In particular, it allowed us to 
show
that, under this hypothesis,  ${\rm gcd}(2,\#G)\cdot \bar\chi(X, \F)=0$ if $d=1$ 
 (\cite{PaCubeInvent}). 
Some  general results were obtained in \cite{PMathAnn} when $d>1$
but  the problem appears to be quite hard. In fact, as is explained in \emph{loc. cit.},
 it is plausible that the statement that $\overline\chi(X, \F)=0$ for $G$ abelian and
 all $X$ of dimension $<\#G$  is equivalent to the truth of Vandiver's conjecture for all prime divisors of 
the order $\#G$. 
In \cite{PMathAnn}, \cite{CPTAnnals}, it was conjectured
that for all $G$, there are integers $N$ (that depends only on $d$) and $\delta$
(that depends only on $\#G$), such that ${\rm gcd}(N, \#G)^\delta\cdot \overline\chi(X, \F)=0$;
this was shown for $G$ with abelian Sylows. 

In this paper, we introduce a new method
 that allows us to handle non-abelian groups.
 We obtain strong results, and,
in the case $d=1$ of curves over $\Z$,  a proof of the above conjecture. 
This is based on Adams-Riemann-Roch 
type identities which are proven using the  K\"unneth formula
and  ``localization" or ``concentration" theorems 
(as given for example by Thomason) in equivariant $\rm K$-theory.
We combine these
  with a study of the action of the Adams--Cassou-Nogu\`es--Taylor operations on $\Cl(\Z[G])$ and
 some classical algebraic number theory.
 
 Write $\#G=2^s3^tm$, with $m$ relatively prime to $6$, and define $a$, $b$ as follows: If the $2$-Sylow subgroups of $G$ have order $\leq 4$,  are cyclic of order $8$, or are dihedral groups, set $a=0$.  If the $2$-Sylow subgroups of $G$ are abelian (but not as in the case above), or  
generalized quaternion,  or  semi-dihedral groups, set $a=1$. 
In all other cases, set $a=s+1$. If the 
$3$-Sylows subgroups of $G$ are abelian, set $b=0$.
In all other cases, set $b=t-1$.
 
\begin{thm}\label{main1}
Let $X\to Y$ be an unramified $G$-cover, with $Y\to \Spec(\Z)$ 
projective and flat of relative dimension $1$. Let $\F$ be a $G$-equivariant coherent 
$\O_X$-module.  For $a$ and $b$ as above,  we have $2^a3^{b}\cdot \overline\chi(X,\F)=0$.
  In particular, if $G$ has order prime to $6$, then $\overline\chi(X,\F)=0$.
\end{thm}

In particular, this implies, in this case of free $G$-action,  the ``Localization Conjecture" of \cite{CPTAnnals} for $d=1$ with
 $N=6$ and $\delta={\mathrm{max}(s,t)}+1$, and significantly
extends the   results of \cite{PaCubeInvent}. For example,  the hypothesis 
``all Sylows are abelian" in  \cite[Theorem 5.2 (a),
Cor. 5.3 and Theorem 5.5 (a)]{PaCubeInvent} can be relaxed. Hence, we obtain 
a ``projective normal integral basis theorem" when $d=1$ and the action is free:
If $X$ is normal and $G$ is such that $a=b=0$, 
then $X\cong {\bf Proj}(\oplus_{n\geq 0} A_n)$
with $A_n$, for all $n>0$, free $\Z[G]$-modules. We can also give a result for
the $G$-module of regular differentials of $X$: Suppose that $X$ is normal;
then $X$ is Cohen-Macaulay and we can consider the relative dualizing sheaf 
$\omega_X$ of $X\to \Spec(\Z)$.
Suppose that 
$\{g_1,\ldots, g_r\}$
is a  set of generators of $G$. Then $\{g_1-1,\ldots, g_r-1\}$ is a set of generators of the augmentation ideal ${\mathfrak g}\subset \Z[G]$; this gives a surjective $\phi: \Z[G]^{r_G}\to {\mathfrak g}$,
where $r_G$ is the minimal number of generators of $G$.
Denote by $\Omega^2(\Z)$ the kernel of $\phi$. By Schanuel's lemma, 
the stable isomorphism class of $\Omega^2(\Z)$ is independent of the choice
of generators. 
As in \cite[Theorem 5.5 (a)]{PaCubeInvent},  we will see that Theorem \ref{main1}   implies:

\begin{thm}\label{mainDiffthm} 
Let $X\to Y$ be an unramified $G$-cover, with $Y\to \Spec(\Z)$ 
projective and flat of relative dimension $1$. Assume that $X$ is normal,
 $X_\Q$ is smooth, and that $G$ acts trivially on ${\rm H}^0(X_\Q, \O_{X_{\Q}})$.
 Assume also that ${\rm H}^1(X, \omega_X)$
is torsion free. Set $h=\dim_\Q{\rm H}^0(X_\Q, \O_{X_\Q})$ and  $g_Y={\rm genus}(Y_\Q)=\dim_{\Q}{\rm H}^0(Y_\Q, \Omega^1_{Y_\Q})$. If $G$ is such that $a=b=0$, then
the $\Z[G]$-module ${\rm H}^0(X, \omega_X)$ is stably isomorphic to $\Omega^2(\Z)^{\oplus h}$.
If in addition, $G$ has odd order and $g_Y>h\cdot r_G$, then we have
\[
 {\rm H}^0(X, \omega_X)\cong \Omega^2(\Z)^{\oplus  h}\oplus\Z[G]^{\oplus (g_Y-h\cdot r_G)}.
\]
 \end{thm}

Now, let us give some more details.
Our technique is $\rm K$-theoretic 
and was inspired by work of M. Nori \cite{NoriRR} and especially of B. K\"ock \cite{KockRRtensor}.
The beginning point is that the Adams-Riemann-Roch theorem for a smooth variety $X$ can be obtained by
combining a fixed point theorem  
for the permutation action of the cyclic group $C_\l=\Z/\l\Z$ 
on the product $X^\l=X\times\cdots \times X$, 
with  the K\"unneth formula (see \cite{NoriRR, KockRRtensor}, here $\l$ is a prime number).
 Very roughly, this goes as follows: 
If $\E$ is a vector bundle over $X$, the K\"unneth formula gives an isomorphism ${\rm R}\Gamma(X, \E)^{\otimes \l}\cong {\rm R}\Gamma(X^\ell, \E^{\boxtimes \l})$. Using localization and the Lefschetz-Riemann-Roch
theorem we can relate the cohomology of the exterior tensor product sheaf
$\E^{\boxtimes \l}$ on $X^\l$ to the cohomology of the restriction  $\E^{\otimes \l}=\E^{\boxtimes \l}{|_{X}}$ 
on
 the $C_\l$-fixed point locus which here is the diagonal $X=\Delta(X)\subset X^\l$. Of course, there is a correction that 
 involves the conormal bundle ${\mathscr {N}}_{X|X^\l}$ of $X$ in $X^\l$. 
 This correction amounts to multiplying $\E^{\otimes \l}$ by $\lambda_{-1}({\mathscr {N}}_{X|X^\l})^{-1}$.
 The multiplier $\lambda_{-1}({\mathscr {N}}_{X|X^\l})^{-1}$ is familiar in Lefschetz-Riemann-Roch type theorems 
 and in this case is given by the inverse  $\theta^\l(\Omega^1_X)^{-1}$ of the $\l$-th Bott class
 of the differentials; 
 the resulting identity eventually gives the Adams-Riemann-Roch theorem for the Adams operation $\psi^\l$. 
 
 In our situation, we have to take great care to 
 explain how enough of this can be done $G$-equivariantly and also for projective $G$-modules without losing 
much information. It also important to connect the $\l$-th tensor powers used above to 
the versions of the $\l$-th Adams operators on the projective class group $\Cl(\Z[G])$ as defined by Cassou-Nogu\`es--Taylor. 
This is all quite subtle and eventually needs to be applied for a
carefully selected  set
of primes $\l$.
Most of our arguments are valid in any dimension $d$ and  we show various 
general Adams-Riemann-Roch type
results. Combining these with a \v Chebotarev density argument, we obtain that  for a $p$-group $G$ with $p$
an odd prime, the obstruction $\overline\chi(X, \F)$ lies in a specific 
sum of  eigenspaces for the action of the Adams--Cassou-Nogu\`es--Taylor
operators on the $p$-power part of the classgroup $\Cl(\Z[G])$
(Theorem \ref{eigencor}). 
When $d=1$ there is only one eigenspace in this sum. We can then see,
using classical results of the theory of cyclotomic fields,
that this eigenspace is trivial
when $p\geq 5$. When $p=3$, results of R. Oliver \cite{OliverD} 
imply that the eigenspace is annihilated by the Artin exponent of $G$. 
The crucial 
number theoretic ingredients are classical results of Iwasawa 
 on cyclotomic class groups and units and the following fact:  
For any $p>2$, both the second and the $p-2$-th Teichm\"uller eigenspaces of the $p$-part of the 
class group of $\Q(\zeta_p)$ are trivial. An argument that  uses again localization  implies
that we can reduce the calculation to $p$-groups. Therefore, by the above, for general $G$, the class  $\overline\chi(X,\F)$ 
in $\Cl(\Z[G])$ is determined by the classes $c_2=\overline\chi(X,\F)$ in $\Cl(\Z[G_2])$ for the cover $X\to X/G_2$
and $c_3=\overline\chi(X,\F)$ in $\Cl(\Z[G_3])$ for $X\to X/G_3$. Here $G_2$, resp. $G_3$, is a $2$-Sylow, resp. $3$-Sylow subgroup of $G$. This then eventually leads to Theorem \ref{main1}.  

It is also important to treat
 (tamely)  ramified covers $X\to Y$ over $\Z$.
 However,   usually the unramified over $\Z$ case is  the hardest one to deal 
with  initially.  After this is done, ramified covers
can be examined by localizing at the 
branch locus  (see for example \cite{CPTAnnals}). This also
occurs here.
In fact, our method applies 
when $\pi: X\to Y$ is tamely ramified. In our situation, the assumption that the ramification is tame 
implies that  the corresponding $G$-cover  $X_\Q\to Y_\Q$ 
given by the generic fibers is unramified.  
 Under certain regularity assumptions, we obtain an Adams-Riemann-Roch 
formula for $\ov\chi(X,\F)$  (Theorem \ref{ARRtame}) which generalizes formulas of Burns-Chinburg \cite{ChinburgBurns} and K\"ock \cite{KockMathAnn} to all dimensions. In this case,
one cannot expect that $\overline\chi(X, \F)$ is annihilated by 
small integers. Nevertheless, when $d=1$, we can obtain 
(almost full) information about the class $ \overline\chi(X,\F)$
from the ramification locus of the cover $X\to Y$. In particular,
under some further  conditions on $X$ and $Y$,
we show that the class ${\rm gcd}(2,\#G)^{v_2(\#G)+2}
{\rm gcd}(3,\#G)^{v_3(\#G)-1}\cdot \overline\chi(X,\F)$
only depends on the restrictions of $\F$ on the local  curves
$X_{\Z_p}$, where $p$
runs over the primes that contain the support of the
ramification locus of $X\to Y$. This proves
  a slightly weakened variant of the ``Input Localization
Conjecture" of \cite{CPTAnnals} for $d=1$
(see Theorem \ref{thminput}). Dealing with wildly ramified covers 
  presents a host of  additional difficulties which are  still not resolved, not even in the
case $d=0$ of number fields.

 \smallskip

 \emph{Acknowledgements:} The author would like to thank T.~Chinburg and M.~Taylor
for  many discussions over the years   and the anonymous referees
for their careful reading and for numerous suggestions that improved the paper. 

\bigskip

\bigskip

\section{Grothendieck groups and Euler characteristics}
 
\subsection{$G$-modules, sheaves and Grothendieck groups}  

\begin{para}
In this paper, modules are left modules while groups act on schemes on the right. Sometimes we
will use the term ``$G$-module"
instead of ``$\Z[G]$-module".  If we fix  a prime number $\l$, we set $\Zl=\Z[\l^{-1}]$,
and in general we denote inverting $\l$ or localizing at the 
multiplicative set of powers of $\l$ by a prime ${}'$. We reserve the notation $\Z_\l$ for the $\l$-adic integers.
Set $C_\l=\Z/\l\Z$ and denote by $S_\l$  the symmetric group
in $\l$ letters. We regard $C_\l$ as the subgroup of $S_\l$ generated by the cycle $(12\cdots \l)$.
\end{para}

\begin{para}\label{idealspara}
If $R=\Z$ or $\Z'$, then the Grothendieck group $\Gr_0(R[G])$
of finitely generated $R[G]$-modules can be identified with the Grothendieck ring
 of finitely generated $R[G]$-modules which are free as $R$-modules (``$R[G]$-lattices")
with multiplication operation given by tensor product. The ring $\Gr_0(R[G])$ is a finite and 
flat $\Z$-algebra. The natural homomorphism $\Gr_0(R[G])\to \scrR_\Q(G):=\Gr_0(\Q[G])$, $[M]\mapsto [M\otimes_R\Q]$,
is surjective with nilpotent kernel (\cite{SwanGro}). Therefore, the prime ideals of $\Gr_0(R[G])$
can be identified with the prime ideals of the representation ring $\scrR_\Q(G)$; these can be described following Segal and Serre (see for example  \cite[\S 4]{CEPTRR}): Suppose that ${\mathfrak p}$ is a prime ideal of $R$. An element $g$ in $G$ is called
${\mathfrak p}$-regular if it is of order prime to the characteristic of $\mathfrak p$; the set of prime ideals of $\scrR_\Q(G)$
that lie above ${\mathfrak p}$ are in $1$-$1$ correspondence with the set of ``$\Gal(\ov\Q/\Q)$-conjugacy classes"
of $\mathfrak p$-regular elements. Here $g$, $g'$ are $\Gal(\ov\Q/\Q)$-conjugate, when there is $t\in \Z$ prime to the exponent ${\rm exp}(G)$ of $G$, such that
$g'$ is conjugate to $g^t$. The prime ideal $\rho=\rho_{(g, \pfr)}$, where $g$ is $\mathfrak p$-regular,
is 
$$
\rho_{(g, \pfr)}=\{ \phi\in \scrR_\Q(G)\ |\ {\rm Tr}(g\, |\, \phi)\in \pfr\}.
$$
\end{para}

\begin{para} We say that a $G$-module $M$ is $G$-{\sl cohomologically trivial} ($G$-c.t.),
or simply {\sl cohomologically trivial} (c.t.) when $G$ is clear from the context,
if for all subgroups $H\subset G$, the cohomology groups ${\rm H}^i(H, M)$
are trivial when $i>0$. If the $\Z[G]$-module $M$ has finite projective
dimension, it is $G$-c.t.
In fact, by \cite[Theorem 9]{AtiyahWall},  the converse is true and $M$ is $G$-c.t. if and only if it has a
resolution
$$
0\to P\to Q\to M\to 0
$$
where $P$, $Q$ are  projective $\Z[G]$-modules.
If $M$ is in addition finitely generated, we can take $P$ and 
$Q$ to also be finitely generated. 
Hence, we can identify
$$
\Gr_0^{\rm ct}(G,\Z)=\Kr_0(\Z[G])
$$
where $\Gr_0^{\rm ct}(G,\Z)$ is the Grothendieck group
of finitely generated $G$-c.t. modules. The following facts can be found in \cite{SwanAnnals},
\cite{MJTclassgroups}:
$\Kr_0(\Z[G])$ is a finitely generated $\Gr_0(\Z[G])$-module by action
given by the tensor product by $\Z[G]$-lattices. The subgroup $\langle \Z[G]\rangle$ of free classes is
a $\Gr_0(\Z[G])$-submodule and so the quotient $\Cl(\Z[G])$ is also a $\Gr_0(\Z[G])$-module.
The  $\Gr_0(\Z[G])$-module structure on $\Cl(\Z[G])$ factors through the quotient 
$\Gr_0(\Z[G])\to \scrR_\Q(G)$.
\end{para}

 \begin{para} For $S$ a scheme (with trivial $G$-action), we will consider 
quasi-coherent sheaves of $\O_S[G]$-modules 
on $S$ and simply call these 
quasi-coherent $\O_S[G]$-modules. We say that a quasi-coherent $\O_S[G]$-module $\F$
is $G$-cohomologically trivial ($G$-c.t.) if, for all $s\in S$, the stalk $\F_s$ is a $G$-c.t. module.
As in \cite[p. 448]{ChinburgTameAnnals}, we can easily see that a quasi-coherent $\O_S[G]$-module $\F$ is $G$-c.t. if and only if 
 for every open affine subscheme $U$ of $S$,
 $\F(U)$ is a $G$-c.t. module.
\end{para}

\subsection{Grothendieck groups and Euler characteristics}

 \begin{para}\label{construction} We refer to \cite{ThomasonEquiv} for general facts 
 about the equivariant $\Kr$-theory of schemes with group action
and for the notations $\Gr_i(G, -)$, resp.  $\Kr_i(G, -)$, of $\Gr$-groups, resp.  $\Kr$-groups,
for coherent, resp. coherent locally free, $G$-equivariant sheaves.
 If $S$ is as above, we denote by $\Gr^{\rm ct}_0(G, S)$ the Grothedieck group of $G$-c.t. coherent
 $\O_S[G]$-modules.
\quash{Similarly, we denote by $\Gr^{\rm pr}_0(G, S)$ the Grothendieck group of coherent
$\O_S[G]$-modules on $S$ which are locally projective.
When $S=\Spec(R)$ and $R$ is Noetherian, by Proposition \ref{resolution}, we can identify $\Gr^{\rm pr}_0(G, S)$ with ${\rm K}_0(R[G])$ by sending the class $[\E]$ of a sheaf $\E$
to the class  of the $R[G]$-module
of its global sections. }

If the structure morphism $g: S\to \Spec(\Z)$ is projective, there is (\cite{ChinburgTameAnnals}, see also below)  a group homomorphism
(the  ``projective equivariant Euler characteristic")
$$
g^{\ct}_*: \Gr^{\rm ct}_0(G, S)\to \Gr^{\rm ct}_0(G, \Spec(\Z))=\Gr^{\rm ct}_0(G,\Z)={\rm K}_0(\Z[G])
$$
where in the target we   identify the class of a sheaf with the class of the 
module of its global sections.
 We will sometimes abuse notation and still write
$g^{\ct}_*$ for the composition of $g^\ct_*$ with 
$\Kr_0(\Z[G])\to \Cl(\Z[G])=\Kr_0(\Z[G])/\langle \Z[G]\rangle$.

Here is a review of the construction of $g^{\ct}_*$.
Suppose that   $\E$ is a coherent   $\O_S[G]$-module with
$G$-c.t. stalks. Following \cite{ChinburgTameAnnals}, we first give a bounded complex $(M^\bullet, d^\bullet) $ of finitely generated projective $\Z[G]$-modules
which is isomorphic to ${\rm R}\Gamma(S, \E)$ in the derived category 
of complexes of $\Z[G]$-modules bounded below: Choose a finite open cover ${\mathcal U}=\{U_j\}_j$ of $S$ by affine 
subschemes $U_j=\Spec(R_j)$. Then all intersections $ U_{j_1\ldots j_m}:=U_{j_1}\cap\cdots \cap U_{j_m}$
are also affine. The sections
$
\E( U_{j_1\ldots j_m})
$
are cohomologically trivial $G$-modules and the usual \v Cech complex $C^\bullet({\mathcal U}, \E)$ is a bounded
complex of cohomologically trivial $G$-modules 
(see the proof of Theorem 1.1 in \cite{ChinburgTameAnnals}).
The   complex $C^\bullet({\mathcal U}, \E)$ is isomorphic to ${\rm R}\Gamma(S, \E)$ and 
since $S\to \Spec(\Z)$ is projective, it has finitely generated  homology groups.
Now the usual procedure, described for example 
in \emph{loc. cit.}, produces a perfect complex $M^\bullet$
of $\Z[G]$-modules (\emph{i.e.} a bounded complex of finitely generated projective $\Z[G]$-modules)
together with a morphism of complexes $M^\bullet\to C^\bullet({\mathcal U}, \E)$ which is 
a quasi-isomorphism. Then $M^\bullet$ is also isomorphic to ${\rm R}\Gamma(S, \E)$
in the derived category and we
define
\begin{equation}
g^{\ct}_*(\E)=\sum\nolimits_i(-1)^i [M^i]
\end{equation}
in $\Gr_0^\ct(G, \Z)=\Kr_0(\Z[G])$; this is independent of our choices.
 \end{para}

\begin{para}
When  in addition the symmetric group $S_\l$ acts
on $S$, we can also consider
the Grothendieck groups $\Gr_0^{\rm ct}(S_\l; G, S)$\quash{, $\Gr_0^{\rm pr}(S_\l; G, S)$}
of $S_l$-equivariant coherent $\O_S[G]$-modules
which are $G$-cohomologically trivial.\quash{, resp. are
locally projective as $\O_S[G]$-modules.}  
If $g: S\to \Spec(\Z)$ is projective and in addition $S_\l$ acts on $S$,
a construction as above gives an Euler characteristic homomorphism
$$
g^{\ct}_*: \Gr^{\rm ct}_0(S_\l;G, S)\to \Gr^{\rm ct}_0(S_\l;G, \Spec(\Z)).
$$
Similarly for $S_\l$ replaced by the cyclic group $C_\l$. For simplicity, we will sometimes omit the superscript ${\ct}$ from
$g^{\ct}_*$ when it is clear from the context.
\end{para}
 
 \begin{para}\label{tamesetup} We now assume that the group $G$ acts on $S$ on the right. We say that $G$ acts tamely on $S$, if for every point $s\in S$,
 the inertia subgroup $I_s\subset G$, which is, by definition, the largest subgroup of $G$
 that fixes $s$ and acts trivially on the residue field $k(s)$, has order prime
 to ${\rm char}(k(s))$. Suppose that the quotient scheme $S/G$ exists and the map
 $\pi: S\to T=S/G$ is finite. Then, by \cite[Lemme 1, Ch. XI]{RaynaudLNM169},  the $G$-cover $\pi$ is, \'etale locally around $\pi(s)$
 on $T$,
 induced from an $I_s$-cover. Hence, if $\F$ is a $G$-equivariant $\O_S$-module, then $\pi_*\F$ is a $G$-c.t.
coherent  $\O_T[G]$-module (see \cite{ChinburgTameAnnals, ChinburgErez}). We then obtain $\pi^{\ct}_*: \Gr_0(G, S)\to \Gr^{\ct}_0(G, T)$ given by $[\F]\mapsto [\pi_*\F]$.
 If $f: S\to  \Spec(\Z)$ is projective, then $S/G$ exists, $\pi$ is finite and $g: T\to \Spec(\Z)$
 is projective (\cite[III, Th. 1]{MumfordAV}). Then the composition $f^{\ct}_*=g^{\ct}_*\cdot \pi^{\ct}_*$ gives the projective equivariant Euler characteristic
 \begin{equation}
 f^{\ct}_*: \Gr_0(G, S)\to \Gr^{\ct}_0(G, \Spec(\Z))=\Kr_0(\Z[G]).
 \end{equation}
 The Grothendieck groups $\Gr_0(G, S)$ and $\Kr_0(\Z[G])$ are $\Gr_0(\Z[G])$-modules
 and the map $f^{\rm ct}_*$ is a $\Gr_0(\Z[G])$-module homomorphism. 
Similarly, for 
 $\Z$ replaced by $\Zl=\Z[\l^{-1}]$.
Sometimes, we will denote $f^{\ct}_*(\F)$ by $\chi(X,\F)$;
then $\overline\chi(X, \F)$ is the image of $\chi(X,\F)$ in $\Cl(\Z[G])=\Kr_0(\Z[G])/\langle \Z[G]\rangle$.  
 
Similarly, if $S_\l\times G$   acts on the projective $f: S\to \Spec(\Z)$ with the subgroup $G=1\times G$ acting tamely, 
we have
\begin{equation}
f^{\ct}_*: \Gr_0(S_\l\times G, S)\to \Gr^{\ct}_0(S_\l; G, \Spec(\Z))
\end{equation}
given as $f_*^{\ct}= g^{\ct}_*\cdot \pi^{\ct}_*$.  
Here $f^{\ct}_*$ is a $\Gr_0(\Z[S_l\times G])$-module homomorphism.
These constructions also work with $\Z$ replaced by $\Zl$
and with $S_\l$ replaced by $C_\l$.
\end{para}

\section{The K\"unneth formula}

\subsection{Tensor powers} 

\begin{para}\label{tensorpar} Let $R$ be a commutative Noetherian ring with $1$.
If  $(M^\bullet, d^\bullet)$ is a bounded chain complex of $R[G]$-modules 
which are flat as $R$-modules, 
we can consider
the total tensor product complex
$$
({M^\bullet}^{\otimes \l},  \partial^\bullet )
$$
whose term of degree $n$ is
$$
({M^\bullet}^{\otimes \l})^n=\bigoplus_{  \stackrel{(i_1,\ldots, i_l)\in \Z^\l}{   i_1+\cdots +i_l=n}  } (M^{i_1} \otimes_R  \cdots \otimes_R M^{i_\l})
$$
with diagonal $G$-action and the differential $\partial^n $ is given by
$$
\partial^n(m_{i_1}\otimes\cdots \otimes m_{i_\l})=\sum_{a=1}^\l (-1)^{i_1+\cdots +i_{a-1}} m_{i_1}\otimes \cdots\otimes m_{i_{a-1}}\otimes d^{i_a}(m_{i_a})\otimes\cdots \otimes m_{i_\l}.
$$
Since the modules $M^i$ are $R$-flat, the complex $({M^\bullet})^{\otimes \l}$ is isomorphic to the $\l$-fold derived tensor  
$$
M^\bullet \buildrel{\rm L}\over{\otimes}_R M^\bullet\buildrel{\rm L}\over{\otimes}_R \cdots \buildrel{\rm L}\over{\otimes}_R M^\bullet
$$
 is the
derived category of complexes of $R[G]$-modules which are bounded above.

 \begin{lemma}\label{tensor}
a)   Let $M_1$, $M_2$ be projective $R[G]$-modules.
Then $M_1\otimes_R M_2$ with the diagonal $G$-action is also a projective $R[G]$-module.

b) Let  $M_1$, $M_2$ be cohomologically trivial $G$-modules which are $\Z$-flat.
Then $M_1\otimes_{\Z} M_2$ with the diagonal $G$-action is also a 
cohomologically trivial $G$-module.
\end{lemma}

\begin{proof}
a) Follows easily from the fact that the tensor product $R[G]\otimes_R R[G]$ with diagonal
$G$-action is $R[G]$-free.  

b) By \cite[Thm. 9]{AtiyahWall} we have resolutions
$0\to Q_i\to P_i\to M_i\to 0$,
with $P_i$, $Q_i$ projective $\Z[G]$-modules.
Using this and (a) we can see that $M_1\otimes_\Z M_2$
has finite projective dimension, hence it is cohomologically trivial.
\end{proof}
\end{para}

\begin{para}\label{permutensorpara}
We  define an action of the symmetric group 
$S_\l$ on the complex $({M^\bullet}^{\otimes \l},  \partial^\bullet )$   as follows
(see \cite[p. 176]{AtiyahAdams}):
$\sigma\in S_\l$ acts on $({M^\bullet}^{\otimes \l})^n=\bigoplus (M^{i_1}\otimes_R   \cdots \otimes_R M^{i_\l})$
by permuting the factors and with the {\sl appropriate sign changes}
so that a transposition of two terms $m_i\otimes m_j$ (where $m_i\in M^i$
and $m_j\in M^j$) comes with the sign $(-1)^{ij}$. (We   see that the action of $\sigma$
commutes with the differentials $\partial^n$.) 
 
Let $M^0$, $M^1$ be finitely generated projective $R[G]$-modules
and let us consider the complex $M^\bullet:=[M^0\xrightarrow{0} M^1]$ (in degrees $0$ and $1$).
Using Lemma \ref{tensor} (a), we  form
the Euler characteristic
$$
\chi((M^\bullet)^{\otimes \l})=\sum\nolimits_{n} (-1)^n\cdot [({M^\bullet}^{\otimes \l})^n]\in \Kr_0(S_\l;G, R),
$$
where $\Kr_0(S_\l;G, R)$ is the Grothendieck group of finitely generated $R[S_l\times G]$-modules
which are $R[G]$-projective.
As in \cite[Prop. 2.2]{AtiyahAdams},  we can see  
that
$\chi((M^\bullet)^{\otimes \l})$ only depends on the class $\chi( M^\bullet )=[M^0]-[M^1]$ in $\Kr_0(R[G])$ and gives a well-defined map, the
``tensor power operation"
$$
\tau^\l: \Kr_0(R[G])\to \Kr_0(S_\l;G, R)\ ;\quad \tau^{\l}([M^0]-[M^1])=\chi((M^\bullet)^{\otimes \l}).
$$
(This statement also follows from \cite{Grayson}, see for example \cite[1]{KockRRtensor}.)
In general, if $M^\bullet$ is a perfect complex of $R[G]$-modules, then as in \cite{AtiyahAdams},
$$
\tau^\l\left(\sum\nolimits_i (-1)^i [M^i]\right)=\sum\nolimits_n (-1)^n [({M^\bullet}^{\otimes \l})^n].
$$
\end{para}

\subsection{The  formula} Suppose that $g: Y\to \Spec(\Z)$ is projective and flat and 
that $\E$ is a coherent $\O_Y[G]$-module which is $G$-c.t. and $\O_Y$-locally free.
Consider the exterior tensor product $\E^{\boxtimes \l}=\otimes_{i=1}^\l p^*_i\E$ of $\E$
on $Y^\l$ with $p_i: Y^\l\to Y$ the $i$-th projection. Then $\E^{\boxtimes \l}$ is a
 $S_\l\times G$-equivariant coherent $\O_{Y^\l}[G]$-module which is $\O_{Y^\l}$-locally free and by Lemma \ref{tensor} (b)
$G$-cohomologically trivial. 
Denote by $g^\l: Y^\l\to \Spec(\Z)$ the structure morphism.

\begin{thm}\label{Kunneththm}(K\"unneth formula)
We have
 \begin{equation}\label{kun}
\tau^\l(g^{\ct}_*(\E))=(g^\l)^{\ct}_*(\E^{\boxtimes \l})
\end{equation}
in $\Kr_0(S_\l;G, \Z)=\Gr^{\rm ct}_0(S_\l; G, \Z)$. 
\end{thm}

\begin{proof} This  follows \cite{KempfRocky}, proof of Theorem 14. Note that 
\cite[Theorem A]{KockRRtensor} gives a similar result for $G=\{1\}$.
Recall the construction of a  bounded complex $(M^\bullet, d^\bullet) $ of finitely generated projective $\Z[G]$-modules
which is quasi-isomorphic to ${\rm R}\Gamma(Y, \E)$ described in \ref{construction} that uses the
 \v Cech complex $C^\bullet({\mathcal U}, \E)$. In this case, all the terms of $C^\bullet({\mathcal U}, \E)$
are $\Z$-flat.
The  complex $C^\bullet({\cal U}, \E)$
is given as the global sections of a corresponding complex ${\mathfrak C}^\bullet({\mathcal U}, \E)$
of quasi-coherent  $\O_Y[G]$-modules   whose terms are direct sums
of sheaves of the form $j_* \E$, where $j: U_{j_1 \ldots  j_m}\hookrightarrow Y$
is the open immersion. Since all the intersections $U_{j_1 \ldots  j_m}$ are affine,  the complex 
${\mathfrak C}^\bullet({\mathcal U}, \E)$ gives an acyclic resolution 
$$
0\to \E\to {\mathfrak C}^\bullet({\mathcal U}, \E)
$$
of the $\O_Y[G]$-module $\E$.
Consider the (exterior) tensor product
$$
{\mathfrak C}^\bullet({\mathcal U}, \E)^{\boxtimes \l}=\bigotimes_{i=1}^\l p^*_i{\mathfrak C}^\bullet({\mathcal U}, \E)
$$
with $S_\l$-action defined following the rule of signs as before.  
Notice that all the terms of  ${\mathfrak C}^\bullet({\mathcal U}, \E)$ have $\Z$-flat stalks.
We can also see that 
${\mathfrak C}^\bullet({\mathcal U}, \E)^{\boxtimes \l}$  is acyclic; thus
it gives an acyclic resolution
$$
0\to \E^{\boxtimes \l}\to {\mathfrak C}^\bullet({\mathcal U}, \E)^{\boxtimes \l}
$$
which respects the $S_\l$-action.
It follows from the definition that the global sections of ${\mathfrak C}^\bullet({\mathcal U}, \E)^{\boxtimes \l}$
are $C^\bullet({\cal U}, \E)^{\otimes \l}$. Hence, we obtain
an isomorphism in the derived category of complexes of $\Z[S_\l\times G]$-modules 
$$
{\rm R}\Gamma(Y^\l, \E^{\boxtimes \l})\xrightarrow{\sim}C^\bullet({\cal U}, \E)^{\otimes \l}.
$$
Using $\phi: M^\bullet\xrightarrow{ }   C^\bullet({\cal U}, \E) $ 
 we also obtain a $\Z[S_l\times G]$-morphism of complexes
$$
 \phi^{\otimes \l}   :  (M^\bullet)^{\otimes \l}\xrightarrow{\ }   C^\bullet({\cal U}, \E)^{\otimes \l}.
$$
Since $\phi $ is a $\Z[G]$-quasi-isomorphism, 
and the terms of $M^\bullet$ and $C^\bullet({\cal U}, \E)$ are $\Z$-flat, 
$\phi^{\otimes \l}$ is a quasi-isomorphism.  
Combining these we get an isomorphism 
in the derived category of complexes of $\Z[S_\l\times G]$-modules
\begin{equation}\label{kunneth}
 (M^\bullet)^{\otimes \l}\xrightarrow{\sim} {\rm R}\Gamma(Y^\l, \E^{\boxtimes \l})
\end{equation}
and by Lemma \ref{tensor} (a),  $(M^\bullet)^{\otimes \l}$ is  perfect 
as a complex of $\Z[G]$-modules.
By taking the Euler characteristics of both sides 
we obtain the result.
\end{proof}
 
  \section{Adams operations}
 
 \subsection{Cyclic powers}
 
\begin{para}\label{411cyclic}
Again $\l$ is a prime and $\Zl=\Z[\l^{-1}]$.
By \cite{KockAdamsCamb} there is  
an   ``Adams operator" homomorphism
$$
 \psi^\l: {\rm K}_0(\Zl[G])\to {\rm K}_0(\Zl[G])
 $$  
defined using ``cyclic power operations" (following constructions of Kervaire and of Atiyah) as follows:
Set $\Delta={\rm Aut}(C_\l)= (\Z/\l\Z)^*$
and consider the semidirect product $C_\l\rtimes \Delta$ given by the tautological
$\Delta$-action on $C_\l$. We view $C_\l\rtimes \Delta$ as a subgroup
of $S_\l={\rm Perm}(C_\l)$ in the natural way. 
Denote by $\sigma$ the generator $1$ of $\Z/\l\Z=C_\l$. If $P$ is a finitely generated projective $\Zl[G]$-module,
the tensor product $ P^{\otimes \l}$ is naturally a $S_\l\times G$-module
which is $\Zl[G]$-projective by Lemma \ref{tensor}. Let $S=\Zl[X]/(X^{\l-1}+X^{\l-2}+\cdots +1)$
and set $z$ for the image of $X$ in $S$. We have $\Zl[C_\l]=\Zl\times 
S$ by $\sigma\mapsto z$. Then $z^\l=1$. The group $\Delta$ acts 
on $S$ via  automorphisms given by $\delta(z)=z^\delta$
 for $\delta\in \Delta=(\Z/\l\Z)^*$. 
 For $a\in \Z/\l\Z$, we set
 $$
 F_a(P^{\otimes \l}):=((S\otimes_\Zl P^{\otimes \l})_a)^\Delta
 $$
 where by definition,
 $$
 (S\otimes_\Zl P^{\otimes \l})_a=\{x\in S\otimes_\Zl P^{\otimes \l}\ |\ \sigma (x)=z^a\cdot x\}.
 $$
 (Notice that $\Delta$ acts on $(S\otimes_\Zl P^{\otimes \l})_a\subset S\otimes_\Zl P^{\otimes \l}$
 by the diagonal action.) By \cite[Cor. 1.4]{KockAdamsCamb}, $F_a(P^{\otimes \l})$ 
 are projective $\Zl[G]$-modules and we have by definition
$$
 \psi^\l([P]):=[F_0(P^{\otimes \l})]-[F_1(P^{\otimes \l})]
 $$
 in ${\rm K}_0(\Zl[G])$.
 \end{para}
 
 \begin{para} Consider now the Grothendieck rings $ \Gr_0(\Zl[C_\l])$, 
$ \Gr_0(\Zl[C_\l\times G])$, etc., where $\l$ is a prime
that does not divide the order $\#G$.
Inflation gives homomorphisms $ \Gr_0(\Zl[C_\l])\to \Gr_0(\Zl[C_\l\times G])$, $\Gr_0(\Zl[G])\to  \Gr_0(\Zl[C_\l\times G])$,
that we will suppress in the notation. Set $v=[\Zl[C_l]]$ for the class 
of the free module in $\Gr_0(\Zl[C_\l])$ and let $\alpha:=v-1$
be the class 
 of the augmentation ideal.
 \end{para}
 
\begin{para}\label{413para}
Write $\Zl[C_\l\times G]=\Zl[G]\times S[G]$. Suppose that $R$ is a $\Zl$-algebra.
If  $N$ is a $R[C_\l\times G]$-module then the $C_\l$-invariants $N^{C_\l}$ give a $R[G]$-module which is a direct summand of $N$. Hence, if $N$ is projective as an $R[G]$-module, resp. is $G$-c.t., then $N^{C_\l}$ is  projective as an $R[G]$-module, resp. is $G$-c.t. The functor
$N\mapsto N^{C_\l}$ from $R[C_\l\times G]$-modules to $R[G]$-modules is exact. Let us consider the homomorphism (cf. \cite[Lemma 4.3, Cor. 4.4]{KockRRtensor})
\begin{equation}
\zeta: \Kr_0(R[C_\l\times G])\to \Kr_0(R[G]), \quad \zeta([N])= \l\cdot [N^{C_\l}]-[N],
\end{equation}
where we subtract the class of $N$ as a $R[G]$-module by forgetting the $C_\l$-action.
This is a $\Gr_0(\Zl[G])$-module homomorphism. 
We can see that $\zeta$ vanishes on the subgroup 
$(v)\cdot {\rm K}_0(R[C_\l\times G])$:
Indeed, suppose that 
 $Q$ is a finitely generated projective $R[C_\l\times G]$-module
  and let us consider the $R[C_\l\times G]$-module 
 $R[C_\l]\otimes_{R}Q\simeq  \Zl[C_\l]\otimes_{\Zl}Q $. Then
 there is an isomorphism (Frobenius reciprocity)
 \begin{equation}\label{Frobenius}
R[C_\l]\otimes_{R}Q
 \simeq R[C_\l]\otimes_{R}P={\rm Ind}^{C_\l\times G}_G(P),
\end{equation}
  where $P={\rm Res}^{C_\l\times G}_{G}(Q)$ is a projective 
 $G$-module with trivial $C_\l$-action. We have
 $$
( R[C_\l]\otimes_{R}Q)^{C_\l}\simeq  (R[C_\l]\otimes_{R}P)^{C_\l}\simeq P
 $$
 and so $\zeta( R[C_\l]\otimes_{R}Q)=\l\cdot [P]-  [P^{\oplus \l}]=0$.
 \end{para}
 
\begin{para} Here $R$ is still a $\Zl$-algebra. Consider also the map
 $$
 \xi: \Kr_0(R[G])\to  \Kr_0(R[C_\l\times G])
 $$
 obtained by inflation (\emph{i.e.} by considering a $G$-module as a $C_\l\times G$-module
 with $C_\l$ acting trivially). This is a $\Gr_0(\Zl[G])$-module homomorphism and we
 can see from the definitions that
 \begin{equation}\label{zeta}
 \zeta\circ \xi=(\l-1)\cdot {\rm id}
 \end{equation}
 as maps $\Kr_0(R[G])\to  \Kr_0(R[G])$.
 \end{para}

\subsection{Cyclic and tensor powers}

\begin{para}
Recall   the tensor power operation
$$
\tau^\l: \Kr_0(\Z'[G])=\Gr^{\rm pr}_0(G, \Zl)\to \Gr^{\rm ct}_0(C_\l; G, \Zl)={\rm K}_0(\Zl[C_\l\times G])
$$
given by the construction in \ref{permutensorpara} applied to $R=\Zl$ 
followed by restriction from $S_\l\times G$ to $C_\l\times G$.

 \begin{prop}\label{Adams}
  For each $x$ in ${\rm K}_0(\Zl[G])$ we have
 \begin{equation*}\label{adams}
 \tau^\l(x)=\psi^\l(x)\quad \hbox{in \ } {\rm K}_0( \Zl[C_\l\times G])/(v)\cdot {\rm K}_0(\Zl[C_\l\times G]).
 \end{equation*}
 \end{prop}
 
 \begin{proof}
 This follows the lines of the proof of \cite{KockRRtensor}, Prop. 1.13. First 
 observe that the argument in \emph{loc. cit.} gives that the map  
 $$
 \bar \tau^\l: {\rm K}_0(\Zl[G])\to {\rm K}_0( \Zl[C_\l\times G])/(v)\cdot {\rm K}_0(\Zl[C_\l\times G])
 $$
 obtained from $\tau^\l$ is a homomorphism. It is then enough to show
 the identity for $x=[P]$, the class of
a finitely generated projective $\Zl[G]$-module $P$.
As above, consider $P^{\otimes \l}$.
 Let $e=\l^{-1}\cdot (\sum_{i=0}^{\l-1}\sigma^i)$ be the idempotent in $\Zl[C_\l]$
 so that $e\cdot P^{\otimes \l}=(P^{\otimes \l})^{C_\l}$. Write 
 $P^{\otimes \l}=(P^{\otimes 
\l})^{C_\l}\oplus Q_1$
 with  $Q_1:=(1-e)\cdot P^{\otimes \l}$ an $S[G]$-module
 (via $\Zl[G]\to S[G]$).
 As in \cite[Ex. 1.5]{KockAdamsCamb} we see that we have
 $$
 F_0(P^{\otimes \l})=e\cdot P^{\otimes \l}=(P^{\otimes \l})^{C_\l}, \ \qquad F_1(P^{\otimes \l})=Q_1^\Delta.
 $$
 There is a short 
 exact sequence
 $$
 0\to Q_1\to Q_1^\Delta\otimes_\Zl \Zl[C_\l]\to Q_1^\Delta\to 0
 $$
 of $C_\l\times G$-modules where the first map is given
 by 
 $$
 q\mapsto \sum_{i=0}^{\l-1}\left(\sum_{a\in \Delta} a\sigma^{-i} \cdot q\right)\otimes \sigma^i
 $$
 and the second map is obtained by tensoring the augmentation map $\Z[C_\l]\to \Z$;
 here $Q^\Delta_1$ is viewed as having trivial $C_\l$-action.
 This gives $[Q_1]=\alpha\cdot [Q_1^\Delta]$ in ${\rm K}_0(\Zl[C_\l\times G])$.
\quash{ (Alternatively, observe that $\Delta$ can be identified with the 
 Galois group of the torsor $\Spec(S)\to \Spec(\Zl)$. The $S$-module $Q_1$
 carries an action of $\Delta$ which is compatible with the $S$-module structure and the action
 of $\Delta$ on $S$. Therefore, it descends to a $\Z'$-module $Q_1^\Delta$ and we have
  $$
 Q_1^\Delta\otimes_{\Zl}S\xrightarrow{\sim} Q_1.
 $$
 Since this isomorphism is canonical it also respects $G$-actions.
 Conclude by noticing that $S$ is also the kernel  $\alpha:={\rm ker}(\Zl[C_\l]\to \Zl)$
of the augmentation map.)}
 This now implies
 $
 [P^{\otimes \l}]=[F_0(P^{\otimes \l})]+\alpha\cdot [F_1(P^{\otimes \l})]
 $
 in ${\rm K}_0(\Zl[C_\l\times G])$, where we regard 
 $F_0(P^{\otimes \l})$, $F_1(P^{\otimes \l})$ as having trivial $C_\l$-action.
 Since $\alpha=v-1$,
 $$
 [P^{\otimes l}]=[F_0(P^{\otimes \l})]-[F_1(P^{\otimes \l})]\quad \hbox{\rm in\ } {\rm K}_0(\Zl[C_\l\times G])/(v)\cdot {\rm K}_0(\Zl[C_\l\times G]),
 $$
 and the result follows.
 \end{proof}
 
 \begin{prop}\label{adamsonfree}
 Let $x_0=[\Zl[G]]$  be the class of the free module  in $\Kr_0(\Zl[G])$.
 We have 
$  \psi^\l (x_0)= x_0$ in $\Kr_0(\Zl[G])$.
 \end{prop}
 
 \begin{proof} Note that \cite[Theorem 1.6 (e)]{KockMathAnn} gives a corresponding result 
 for the (a priori different) Adams operators defined via exterior 
powers.
 Set $\Gamma=C_\l\rtimes(\Z/\l\Z)^*\subset S_\l={\rm Perm}({\mathbb F}_\l)$;
 the element $\gamma=(\sigma^a, b)$ is then the 
 (affine)
map ${\mathbb F}_\l\to {\mathbb F}_\l$ given by $\gamma(x)= bx+a$. Consider the  set $G^\l={\rm Maps}( {\mathbb F}_\l, G)$ with 
$G\times\Gamma$--action given by $(g,\gamma)\cdot (g_x)_{x\in {\mathbb F}_\l}=(gg_{\gamma^{-1}(x)})_{x\in {\mathbb F}_\l}$. 
Suppose $(g,\gamma)$ stabilizes $(g_x)_{x\in {\mathbb F}_\l}$
so that $gg_{x}=g_{\gamma(x)}$ for all $x\in   {\mathbb F}_\l$.
If $b\neq 1$, there is $y\in {\mathbb F}_\l$ such that $\gamma(y)=y$.
Then $gg_y=g_y$ and so $g=1$. If $b=1$, 
we have 
$gg_{x}= g_{x+a}$ for all $x\in   {\mathbb F}_\l$; this gives $g^\l=1$
and so again $g=1$, since $\l$ is prime to $\#G$. We conclude that the stabilizer in  $G\times\Gamma$ of 
any $\und g=(g_x)_{x\in {\mathbb F}_\l}\in G^\l$ lies in $1\times \Gamma$. Therefore, 
$G^\l$ is in $G\times\Gamma$--equivariant bijection with a disjoint union of sets of the form $G\times (\Gamma/\Gamma_{\und g})$ 
with $\Gamma_{\und g}$ the stabilizer subgroup in $\Gamma$. Hence, the 
$\Zl[G\times\Gamma]$-module $\Zl[G]^{\otimes\l}=\Zl[G^\l]$ is isomorphic to a direct sum
of modules of the form $\Zl[G]\otimes_{\Z'}\Zl[\Gamma/\Gamma_{\und g}]$.
By the definition of $F_a$ (\cite[\S 1]{KockAdamsCamb}, or \S \ref{411cyclic}) we see that $F_a(\Zl[G]\otimes_{\Z'}\Zl[\Gamma/\Gamma_{\und g}])\simeq \Zl[G]\otimes_{\Zl} F_a(\Zl[\Gamma/\Gamma_{\und g}])$. We conclude that $F_a(\Zl[G]^{\otimes\l})$ are free $\Zl[G]$-modules.
Therefore, $\psi^\l (x_0)= m\cdot x_0$ in $\Kr_0(\Zl[G])$, for some $m\in \Z$,
and by comparing $\Zl$-ranks (for this we may assume $G=\{1\}$) we can easily see that $m=1$.
\end{proof}
 \end{para}

 \begin{para}
 If $Z$ is a projective flat $G$-scheme over $\Spec(\Zl)$, there is an Adams operation
 $\psi^\l: \Kr_0(G, Z)\to \Kr_0(G, Z)$ (\cite{KockGRR}). By an argument as above, or by using the equivariant splitting principle as in 
 \cite{KockRRtensor}, we can see that 
 \begin{equation}\label{424}
 \psi^\l(\F)=[\F^{\otimes \l}]  \end{equation}
 in the quotient  $\Kr_0(C_\l\times G, Z)/(v)\cdot \Kr_0(C_\l\times G, Z)$.
 \end{para}
 
 \subsection{The Cassou-Nogu\`es--Taylor Adams operations}
 
 \begin{para}
 We continue to assume that $\l$ is prime to the order $\# G$. Then, by \cite{SwanAnnals},
 finitely generated projective $\Z'[G]$-modules are locally free.
 
 If $(n, \#G)=1$, we denote by $\psi^{\CNT}_n: \Cl(\Z[G])\to \Cl(\Z[G])$ the  Adams operator homomorphism  defined by Cassou-Nogu\`es-Taylor (\cite{CNTAdams},
\cite{MJTclassgroups}).
(Roughly speaking, this is  given, via the Fr\"ohlich description,
 as the dual of  the Adams operation $\psi^n(\chi)(g):=\chi(g^n)$ on the character group.)
 The operators $\psi^{\CNT}_n$ also restrict to operators on $\Cl(\Zl[G])$;
 we will denote these by the same symbol.  
 Note that we can identify  $\Cl(\Zl[G])$ with the subgroup of $\Kr_0(\Zl[G])$ of elements of rank $0$.
 Denote by $r: {\rm K_0}(\Zl[G])\to \Z$ the rank homomorphism.

K\"ock has shown that the Cassou-Nogu\`es-Taylor 
Adams operators can be described in terms of (arguably more natural) Adams operators $\psi^\l_{\rm ext}$ defined via exterior 
powers and the Newton polynomial (\cite[Theorem 3.7]{KockMathAnn}). Here, we use 
his arguments
to obtain a similar relation with the Adams operators $\psi^\l$
of \cite{KockAdamsCamb} defined via cyclic powers 
which are better suited to our application.

\begin{prop}
\label{KockvsCNT}
 Let $\l'$ be a prime with $\l \l'\equiv 1\mod {\rm exp}(G)$
and set $x_0=[\Zl[G]]$ for the class of the free module
of rank $1$.
Then we have
\begin{equation}\label{adamscompare}
 \psi^\l(x-r(x)\cdot x_0)=\l\cdot \psi^{\rm CNT}_{\l'}(x-r(x)\cdot x_0)
\end{equation}
for  all $x\in {\rm K_0}(\Zl[G])$.
\end{prop}
  
 \begin{proof}  
Let us explain how we can deduce this by combining results and arguments 
from   \cite{KockAdamsPolish},  \cite{KockAdamsCamb}, \cite{KockMathAnn}.
Since $\psi^\l: {\rm K_0}(\Zl[G])\to {\rm K_0}(\Zl[G])$ 
is a group homomorphism (\cite[Prop. 2.5]{KockAdamsCamb})
that preserves the rank, $\psi^\l$ restricts to an (additive) operation on 
the subgroup $\Cl(\Zl[G])$ of rank $0$ elements. The group 
${\rm K_0}(\Zl[G])$ is generated by the classes of  
 locally free  left ideals in $\Z'[G]$; hence, we can assume $x=[P]$ where
 $P$ is such a ideal.
 We may assume that  $P\otimes_{\Zl}\Z_v= \Z_v[G]\cdot \lambda_v$, with $\lambda_v\in  \Q_v[G] ^*\cap  \Z_v[G]$,
 so that  $P=\bigcap_{v\neq l}(\Z_v[G]\cdot \lambda_v\cap \Q[G]) $.
Then a Fr\"ohlich representative of $x-x_0$ is given via 
 the classes (``reduced norms") $[\lambda_v]\in \Kr_1(\Q_v[G])$ of $\lambda_v\in \Q_v[G]^*$.
 Note that we have (\cite{MJTclassgroups})
 \begin{equation}\label{frohomo}
 \Kr_1(\Q_v[G])\simeq {\rm Hom}_{{\rm Gal}(\bar\Q_v/\Q_v)}(\Kr_0(\bar\Q_v[G]), \bar\Q_v^*).
\end{equation}
By \cite{KockAdamsCamb},  cyclic power Adams operators $\psi^\l$ can also be defined 
on the higher $\Kr$-groups  $\Kr_i(R[G])$, $i\geq 1$, for every commutative $\Zl$-algebra. In particular, we have $\psi^\l$ on $\Kr_1(\Z_v[G])$, 
$\Kr_1(\Q_v[G])$, for $v\neq (\l)$, and on $\Kr_1(\Q[G])$.  Using \cite[Cor. 1.4 (c)]{KockAdamsCamb}, 
we see that the base change
homomorphism
$$
\Kr_1(\Z_v[G])\to \Kr_1(\Q_v[G])
$$
commutes with   $\psi^\l$. In fact, by \cite[\S 3]{KockAdamsCamb}
the operators $\psi^\l$ are  defined via the cyclic operations $[a]_\l$ of \emph{loc. cit.} as
$\psi^\l=[0]_\l-[1]_\l$. Moreover, the operations $[a]_\l$ 
are given  via continuous maps on the level of spaces
that give $\Kr$-theory in the style of Gillet-Grayson (see  \emph{loc. cit.}). Using this we can see that  the 
topological argument of the proof of \cite[Prop. 3.1]{KockMathAnn} also applies to the operators $[a]_\l$ and $\psi^\l$ and so we obtain the commutative diagram of \cite[Prop. 3.1]{KockMathAnn}
for $K=\Q$, $\mathfrak p=(v)$ and $\gamma=[a]_\l$ or $\psi^\l$, \emph{i.e.} that $[a]_\l$ and $\psi^\l$ commute with the connecting homomorphism $\Phi:  \Kr_1(\Q_v[G])\to \Kr_0(\Zl[G])$.
As in the proof of 
\cite[Cor. 3.5]{KockMathAnn}  this, together with the localization sequence,
implies that 
 the element $\psi^\l(x-x_0)=\psi^\l(x)-\psi^\l(x_0)$
 has Fr\"ohlich representative given by
 $
 (\psi^\l([\lambda_v]))_v
 $
 with $\psi^\l: \Kr_1(\Q_v[G])\to \Kr_1(\Q_v[G])$ as above
(see \cite{KockMathAnn} for more details).

Now by
 \cite[Cor. (c) of Prop. 1]{KockAdamsPolish} the operator $\psi^\l$ on $\Kr_1(\Q_v[G])$ and  $\Kr_1(\Q[G])$ agrees with the (more standard) Adams operator $\psi^\l_{\rm ext}$
defined using exterior powers and the Newton polynomial (as for example in 
\cite{KockMathAnn}). 
In fact, then by \cite[Cor. 1 of Theorem 1]{KockAdamsPolish}
(or the proof of \cite[Theorem 3.7]{KockMathAnn}),
 $
 \psi^\l([\lambda_v]) 
 $ is  given via (\ref{frohomo}) by $\chi\mapsto ([\lambda_v](\psi^{\l'}(\chi)))^\l $.
The result now follows  from  the Fr\"ohlich description 
of $\Cl(\Zl[G])$ (\cite{MJTclassgroups}) and  the definition of the Cassou-Nogu\`es-Taylor 
Adams operator.
\end{proof}
\end{para}

\begin{Remark}
{\rm Since $\psi^\l$ is only defined for $\Zl=\Z[\l^{-1}]$-algebras the
above proposition does not give an expression for the Cassou-Nogu\`es-Taylor
Adams operators on $\Cl(\Z[G])$. Still, this weaker result 
is enough for our purposes.}
\end{Remark}

 \section{Localization and Adams-Riemann-Roch identities}

\subsection{Localization for $C_\l$- and $C_\l\times G$-modules}

\begin{para}
For simplicity, we set $\scrR(C_\l)=\Gr_0(\Zl[C_\l])'=\Gr_0(\Zl[C_\l])[\l^{-1}]$, 
$\scrR(C_\l\times G)=\Gr_0(\Zl[C_\l\times G])'$, etc., where $\l$ is a prime
that does not divide the order $\#G$.
Inflation gives homomorphisms $\scrR(C_\l)\to \scrR(C_\l\times G)$, $\scrR(G)\to  \scrR(C_\l\times G)$,
that we will suppress in the notation.
Denote by $\alpha$ the class in $\scrR(C_\l)$
 of the augmentation ideal of $\Zl[C_\l]$ and set $v=1+\alpha=[\Zl[C_\l]]$.
 The ring structure in $\scrR(C_\l)$ is such that $v^2=\l\cdot v$.
 Also denote by $I_G$ the ideal of $\scrR(G)$ given as the kernel
 of the rank homomorphism. Set
 $$
 \scrR(C_\l\times G)^\wedge:=\varprojlim\nolimits_{n} \scrR(C_\l\times G)/(I^n_G \scrR(C_\l\times G)+v\scrR(C_\l\times G)).
 $$
 In general, if $M$ is an $\scrR(C_\l\times G)$-module, we set $M^\wedge:=\varprojlim_{n} M/(I^n_G M+vM)$
 which is a $\scrR(C_\l\times G)^\wedge$-module.
The maximal ideals of $\scrR(C_\l\times G)^\wedge$ correspond to maximal ideals of $\scrR(C_\l\times G)$ that contain
$I_G \scrR(C_\l\times G)+v\scrR(C_\l\times G)$; we can see (\emph{cf.} \ref{idealspara}) that these are the maximal ideals of the form $\rho_{((\sigma, 1), (q))}$
with $\l\neq q$ and $\sigma$ is a generator of $C_\l$. (The ideal $\rho_{((\sigma, 1), (q))}$ is independent of the choice of the generator $\sigma$,
there is exactly one ideal for each prime $q \neq \l$.) If $\phi: M\to N$ is an $\scrR(C_\l\times G)$-module
homomorphism such that the localization $\phi_\rho: M_\rho\to N_\rho$ is an isomorphism for every $\rho$
of the form $\rho_{((\sigma, 1), (q))}$ with $q\neq \l$ as above, then the induced $\phi^\wedge: M^\wedge \to N^\wedge$ is an isomorphism of $\scrR(C_\l\times G)^\wedge$-modules.
\end{para}

\subsection{Cyclic localization on products}\label{cyclicloc}

\begin{para}
Fix a prime $\l$ which does not divide $\#G$ and as before  set $\Zl=\Z[\l^{-1}]$.
Suppose that $Z\to \Spec(\Zl)$ is a quasi-projective scheme equipped with an action of $G$.
We will consider 
``localization" on the fixed points for the action of the cyclic group $C_\l$ on the $\l$-fold fiber product $Z^\l$ over 
$\Spec(\Z')$. In fact, the product $C_\l\times G$ acts on $Z^\l$;  $C_\l$ acts by permutation of the factors and
$G$ acts diagonally.
\end{para}

\begin{para}\label{cyclicloc2} Consider a maximal ideal $\rho$ of $\scrR(C_\l\times G)$ of the form $\rho_{((\sigma, 1), (q))}$
with $q\neq \l$ and $\sigma$ a generator of $C_\l$ as before. The corresponding fixed point subscheme $Z^\rho$ of $Z^\l$ is by definition the reduced union of the translates of the fixed
subscheme $Z^{(\sigma, 1)}$ of the element $(\sigma, 1)$. In our case, this is the diagonal:
 \begin{equation}
(Z^\l)^{\rho}=(Z^\l)^{(\sigma, 1)}\cdot (C_\l\times G)=\Delta(Z)=Z\hookrightarrow Z^\l.
\end{equation}
Consider the homomorphism obtained by push-forward of coherent sheaves
$$
\Delta_*: \Gr_0(C_\l\times G, Z)\to \Gr_0(C_\l\times G, Z^\l).
$$
The ``concentration" theorem   \cite[Theorem 6.1]{CEPTRR} implies that after localizing at 
any  $\rho$ as above,  we obtain an $\scrR(C_\l\times G)_\rho$-module isomorphism
$$
(\Delta_*)_\rho: \Gr_0(C_\l\times G, Z)_\rho\xrightarrow{\sim} \Gr_0(C_\l\times G, Z^\l)_\rho.
$$
It now follows as above that $\Delta_*$ gives a $\scrR(C_\l\times G)^\wedge$-module isomomorphism
\begin{equation}
\Delta^\wedge_*: \Gr_0(C_\l\times G, Z)^\wedge \xrightarrow{\sim} \Gr_0(C_\l\times G, Z^\l)^\wedge.
\end{equation}
Here the completions are given as above for the $\scrR(C_\l\times G)$-modules
$\Gr_0(C_\l\times G, Z)'$ and $\Gr_0(C_\l\times G, Z^\l)'$.
 
\begin{Definition}
Let 
$
L_\bullet: \Gr_0(C_\l\times G , Z^{\l})\to \Gr_0(C_\l\times G, Z)^\wedge
$
be the $\scrR(C_\l\times G)$-homomorphism defined as the composition of   $\Gr_0(C_\l\times G, Z^{\l})\to 
 \Gr_0(C_\l\times G, Z^{\l})^\wedge$ with the inverse of $\Delta_*^\wedge$.
Set
$
\al:=L_\bullet(1)\in \Gr_0(C_\l\times G, Z)^\wedge.
$
Then 
\begin{equation}\label{wedge1}
\Delta^\wedge_*(\al)=1=[\O_{Z^{\l}}].
\end{equation}
\end{Definition}
\end{para}

 \begin{para}
 If $\F$ is a $G$-equivariant locally free coherent $\O_Z$-module  on $Z$, then $\F^{\boxtimes \l}$
 is a $C_\l\times G$-equivariant locally free coherent $\O_{Z^{\l}}$-module on $Z^{\l}$ and $\F^{\otimes \l}\simeq \Delta^*(\F^{\boxtimes \l})$ as $C_\l\times G$-sheaves on $Z$. Using (\ref{wedge1}) and the projection formula we obtain
\begin{equation*}
 \F^{\boxtimes \l}=\Delta^\wedge_*(\al)\otimes \F^{\boxtimes \l}=
\Delta^\wedge_*(\al\otimes \Delta^*(\F^{\boxtimes \l}))=\Delta^\wedge_*(\al\otimes \F^{\otimes \l}).
\end{equation*}
Therefore,
\begin{equation}\label{wedge2}
\F^{\boxtimes \l}=\Delta^\wedge_*(\al\otimes \F^{\otimes \l})
\end{equation}
in $\Gr_0(C_\l\times G, Z^{\l})^\wedge$.
\end{para}

\begin{para} Suppose that $Z$ is smooth over $\Spec(\Zl)$, then $Z^{\l}$ is also smooth.
Denote by ${\mathscr N}_{Z|Z^{\l}}$ the locally free conormal bundle ${\mathscr I}_Z/{\mathscr I}_{Z}^2$ of $Z\subset Z^{\l}$
which gives a class in $\Kr_0(C_\l\times G, Z)$. (Here   ${\mathscr I}_Z$ is the ideal sheaf of $Z\subset Z^{\l}$).
As in the proof of the  Lefschetz-Riemann-Roch theorem (\cite{ThomasonLRR}, \cite{CEPTRR}), 
we can see using the self-intersection formula that the homomorphism $L_\bullet$ is given as the composition
of the restriction 
$$
\Gr_0(C_\l\times G, Z^{\l})\cong \Kr_0(C_\l\times G, Z^{\l})\xrightarrow {\Delta^*} \Kr_0(C_\l\times G, Z)\cong \Gr_0(C_\l\times G, Z)
$$
 followed by 
multiplication by $\lambda_{-1}({\mathscr N}_{Z|Z^{\l}})^{-1}\in \Kr_0(C_\l\times G, Z)^\wedge$,
and so $\al$ is the image of $\lambda_{-1}({\mathscr N}_{Z|Z^{\l}})^{-1}$ in $\Gr_0(C_\l\times G, Z)^\wedge$. By \emph{loc. cit.} 
$$
\lambda_{-1}({\mathscr N}_{Z|Z^{\l}}):=\sum_{i=0}^{\rm top}(-1)^i[\wedge^i{\mathscr N}_{Z|Z^{\l}}]
$$
is invertible in all the localizations $\Kr_0(C_\l\times G, Z)_\rho$, with $\rho$ as above,
so it is invertible in $\Kr_0(C_\l\times G, Z)^\wedge$.
\end{para}

\begin{para}\label{nonsmoothlambda} 
If $Z$ is projective but not necessarily smooth, we can 
describe $L_\bullet$ by following \cite{BFQ, Quart}: Embedd $Z$ as a closed $G$-subscheme of a smooth projective bundle $P={\mathbb P}(\Gg)\to \Spec(\Z')$ 
with $G$-action (\cite{KockGRR}). Then $P^{\l}$ is also smooth over $\Spec(\Z')$.
Let us write $i: Z\hookrightarrow P$, $i^\l: Z^{\l}\hookrightarrow P^{\l}$, for the corresponding closed immersions.  Starting with  $x\in \Gr_0(C_\l\times G, Z^{\l})$ we first ``resolve $x$ on $P^{\l}$", \emph{i.e.} represent the push-forward $(i^\l)_*(x)$ by
a bounded complex $\E^\bullet(x)$ of $C_\l\times G$-equivariant locally free coherent sheaves  on $P^{\l}$ which is exact off $Z^{\l}$
(so that the homology of $\E^\bullet(x)$ gives back $x$). Next, we restrict the complex $\E^\bullet(x)$ to $P$ to obtain $\E^\bullet(x)_{|P}$,  a complex exact off $Z=P\cap Z^{\l}$. Finally, we take the class $[h(\E^\bullet(x)_{|P})]$ of the  homology of $\E^\bullet(x)_{|P}$
to obtain an element of $\Gr_0(C_\l\times G, Z)$ (\emph{cf.} \cite{SouleBook}).
For simplicity, write $T_\l=C_\l\times G$.
Then $x\mapsto [h(\E^\bullet(x)_{|P})]$ is the composition
\begin{equation}
\Gr_0(T_\l, Z^{\l})\xrightarrow{\buildrel{h^{-1}}\over\sim } \Kr^{Z^{\l}}_0(T_\l, P^{\l})\xrightarrow{ |P}\Kr^{P\cap Z^{\l}}_0(T_\l, P)\xrightarrow{h} \Gr_0(T_\l, Z),
\end{equation}
where in the middle we have the relative $\Kr$-groups of complexes of $T_\l$-equivariant  locally free
sheaves exact off $Z^{\l}$, resp. $Z^{\l}\cap P=Z$ (see \cite[\S 3]{SouleBook}
and
\cite[Definition 2.1]{BFQ} for more details).
Finally, we multiply $[h(\E^\bullet(x)_{|P})]$  by  the restriction 
$\lambda_{-1}({\mathscr N}_{P|P^{\l}})^{-1}_{|Z}$
in  $\Kr_0(C_\l\times G, Z)^\wedge$  of the inverse $\lambda_{-1}({\mathscr N}_{P|P^{\l}})^{-1}\in   \Kr_0(C_\l\times G, P)^\wedge$.
We claim that 
\begin{equation}\label{5.2.12}
L_\bullet(x)=[h(\E^\bullet(x)_{|P})]\cdot \lambda_{-1}({\mathscr N}_{P|P^{\l}})^{-1}_{|Z}.
\end{equation}
To verify  (\ref{5.2.12}), we can follow the argument in the proof of Lemma 1 in \cite{Quart}
which applies in this situation. The reader is referred to \emph{loc. cit.} for more details. 
\end{para}
 
\subsection{Adams-Riemann-Roch identities}

\begin{para} Let $f: X\to \Spec(\Z)$ be projective and flat with a tame action of $G$.
Then the quotient scheme $\pi: X\to Y=X/G$ exists and $\pi$ is finite. Choose a prime $\l$
with 
$(\l, \#G)=1$. We will apply the set-up
of the previous section to $Z=X'=X\times_{\Spec(\Z)}\Spec(\Z')$.

We have the projective equivariant Euler characteristics
$$
f_*: \Gr_0(C_\l\times G, X')'\to \Kr_0(\Zl[C_\l\times G])'
$$
 and   similarly $f^{\l}_*:
 \Gr_0(C_\l\times G, (X')^\l)'\to \Kr_0(\Zl[C_\l\times G])'$. (Here, we omit the superscript $\ct$ from the notations $f_*^{\ct}$ and $(f^l_*)^{\ct}$.)
These are both $\scrR(C_\l\times G)=\Gr_0(\Z'[C_\l\times G])'$-module
homomorphisms  
and induce $\scrR(C_\l\times G)^\wedge$-homomorphisms
\begin{eqnarray*}
f^\wedge_* :  \Gr_0(C_\l\times G, X')^\wedge \to \Kr_0(\Zl[C_\l\times G])^\wedge,  \\
(f^{\l}_*)^\wedge :  \Gr_0(C_\l\times G, (X')^\l)^\wedge\to \Kr_0(\Zl[C_\l\times G])^\wedge.
\end{eqnarray*}
Now suppose $\F$ is a $G$-equivariant coherent locally free $\O_X$-module.
We apply $(f^{\l}_*)^\wedge $ on both sides of  (\ref{wedge2}). Using 
$f^{\l}_*\cdot \Delta_*=f_*$ we obtain 
\begin{equation}\label{ARRproof1}
(f^{\l})^\wedge_*(\F^{\boxtimes \l})=f^\wedge_*(\al\otimes \F^{\otimes \l}) 
\end{equation}
in $\Kr_0(\Zl[C_\l\times G])^\wedge$.
\end{para}

\begin{para}
Assume in addition that $\pi$ is flat, then so is $\pi^\l: X^\l\to Y^\l$. 
Set $\E=\pi_*\F$. This is a 
$G$-c.t. coherent  $\O_Y[G]$-module which is $\O_Y$-locally free.
We have 
  $\pi^\l_*(\F^{\boxtimes \l})=\E^{\boxtimes \l}$ on $Y^{\l}$.
The K\"unneth formula (\ref{kun}) now gives 
\begin{equation*}
f^{\l}_*( \F^{\boxtimes \l})=
g^{\l}_*( \pi^\l_*(\F^{\boxtimes \l}))=g^{\l}_*(\E^{\boxtimes \l})=\tau^{\l}(g^\ct_*(\E))=\tau^{\l}(f^\ct_*(\F)).
\end{equation*}
Combining this with (\ref{ARRproof1}) gives
\begin{equation*}
\tau^{\l}(f^\ct_*(\F))=f^\wedge_*(\al \otimes \F^{\otimes \l})
\end{equation*}
in $\Kr_0(\Zl[C_\l\times G])^\wedge$. Using Proposition \ref{Adams} we   obtain:
\begin{equation}\label{ARRproof4}
\psi^\l(f^\ct_*(\F))=f^\wedge_*(\al \otimes \F^{\otimes \l}).
\end{equation}
 Finally, using (\ref{424}), this gives the Adams-Riemann-Roch identity
 \begin{equation}\label{ARR6}
\psi^\l(f^\ct_*(\F))=f^\wedge_*(\al \otimes \psi^\l(\F))
\end{equation}
in $\Kr_0(\Zl[C_\l\times G])^\wedge$.
\end{para}

 \begin{para}
Let us now consider $\zeta: \Gr_0(C_\l\times G, Z)'\to \Gr_0(G, Z)'$ given by $\F\mapsto \l\cdot [\F^{C_\l}]-[\F]$
(\emph{cf.}  \ref{413para}).
This is an $\scrR(G)=\Gr_0(\Z'[G])'$-module homomorphism.
As in \ref{413para}, we can see that $\zeta$ vanishes on $(v)\cdot \Gr_0(C_\l\times G, Z)'$
using Frobenius reciprocity. Therefore, 
  it also gives 
$$
\zeta^\wedge:  \Gr_0(C_\l\times G, Z)^\wedge\to  \Gr_0(G, Z)^\wedge=\varprojlim\nolimits_n \Gr_0(G, Z)'/I_G^n\cdot \Gr_0(G, Z)'.
$$
\begin{thm} Under the above assumptions, we have
\begin{equation}\label{ARRzeta}
(\l-1)\psi^\l(f^\ct_*(\F))=f^\wedge_*(\zeta^\wedge(\al)\otimes \psi^\l(\F))
\end{equation}
in   
$
\Kr_0(\Zl[G])^\wedge=\varprojlim\nolimits_n\Kr_0(\Zl[G])'/I^n_G\cdot \Kr_0(\Zl[G])'.
$
\end{thm}
\begin{proof}
Apply to the identity (\ref{ARR6}) the natural map induced on the completions by the map $\zeta$ of \ref{413para}.
The result then follows by using (\ref{zeta}).
 \end{proof}

 Under some additional assumptions, we will
see in the next section that $\al=L_\bullet(1)$ is given as the inverse of a Bott element.
This justifies calling (\ref{ARR6}), (\ref{ARRzeta}),   Adams-Riemann-Roch identities. 

\end{para}
 
\subsection{Localization and Bott classes}

\begin{para} We return to the more general set-up of \S \ref{cyclicloc}.
Suppose in addition that $f: Z\to \Spec(\Zl)$ is a local complete intersection. Then, we can find a $\Zl[G]$-lattice
$E$   such that $f$ factors $G$-equivariantly $Z \hookrightarrow  {\mathbb P}(E)\to \Spec(\Zl)$,
where $P={\mathbb P}(E)$ is the projective space with linear $G$-action determined by $E$.
In this,  the first morphism $i: Z \rightarrow  {\mathbb P}(E)$ is a closed immersion,
and the conormal sheaf ${\mathscr N}_{Z|P}:={\mathscr I}_Z/{\mathscr I}^2_Z$ is a $G$-equivariant sheaf
locally free over $Z$ of rank equal to the codimension $c$ of $Z$ in $P$
(see \cite[\S 3]{KockGRR}). 
By definition, the cotangent element of $Z$ is $T^\vee_Z:=[i^*\Omega^1_{P}]-[{\mathscr N}_{Z|P}]$
in $\Kr_0(G, Z)$; it is independent of the choice of such an embedding.
The following result appears in \cite{KockRRtensor}  if $Z$ is smooth and $G$ acts trivially on $Z$; it is inspired by an observation of M.~Nori \cite{NoriRR}.
(The case that $Z$ is smooth and $G$ acts trivially is enough for the proof of our main result 
in the unramified case, Theorem \ref{main1}.)

\begin{thm}\label{Bottthm}
Suppose that $Z\to\Spec(\Zl)$ is a projective scheme with $G$-action
which is a local complete intersection. Then the element $\al=L_\bullet(1)\in \Gr_0(C_\l\times G, Z)^\wedge$ is the image of the inverse  $\theta^{\l}(T^\vee_Z)^{-1} \in \Kr_0(G, Z)^\wedge$ of   the Bott class 
 under the natural homomorphism $\Kr_0(G, Z)^\wedge\to \Gr_0(C_\l\times G, Z)^\wedge$.
\end{thm}

Here the inverse of the Bott class $\theta^{\l}(T^\vee_Z)^{-1}=\theta^{\l}(i^*\Omega^1_{P})^{-1}\cdot \theta^{\l}({\mathscr N}_{Z|P})$
is defined   in the $I_G$-adic completion $\Kr_0(G, Z)^\wedge$ as in \cite[p. 432]{KockGRR}. As is remarked in \emph{loc. cit.}, 
if $G$ acts trivially on $Z$, then the completion is not needed: The Bott class 
$\theta^{\l}(T^\vee_Z)$ is then defined and is invertible in 
$\Kr_0(Z)'=\Kr_0(Z)[\l^{-1}]$.

\begin{proof}
 Starting with $i: Z\hookrightarrow P$ as above, we obtain a similar embedding of the fibered product 
$
i^\l: Z^{\l}\hookrightarrow P^{\l}
$
which also makes $Z^{\l}$ a local complete intersection in the smooth $P^{\l}$.
We now use \ref{nonsmoothlambda} to calculate $\al=L_\bullet(1)$.
Since $P\to \Spec(\Zl)$ is smooth and projective, we can find a $G$-equivariant  resolution $\E^\bullet \to i_*\O_Z$ 
of $i_*\O_Z$ by a bounded complex of $G$-equivariant locally free coherent $\O_P$-sheaves on $P$. Then, the total 
exterior tensor product $(\E^\bullet)^{\boxtimes \l}=\otimes_{j=1}^{\l} p^*_j\E^\bullet $ with its natural $C_\l$-action 
(with rule of signs as in \ref{tensorpar}) provides 
a $C_\l\times G$-equivariant locally free resolution of the $C_\l\times G$-coherent sheaf
  $(i^\l)_*(\O_{Z^{\l}})\simeq (i_*\O_Z)^{\boxtimes \l}=\otimes_{j=1}^{\l} p^*_j (i_*\O_Z)$  on $P^{\l}$.  
Let us restrict to $P$, \emph{i.e.} pull back the complex $(\E^\bullet)^{\boxtimes \l}$ via the diagonal 
  $\Delta_P: P\hookrightarrow P^{\l}$. We obtain the total $\l$-th tensor product 
  \begin{equation}
 ( \E^\bullet)^{\otimes \l}=(\E^\bullet)^{\boxtimes \l}|_P=\Delta^*_P((\E^\bullet)^{\boxtimes \l})
  \end{equation}
  with its natural $C_\l\times G$-action.
  Since $\E^\bullet$ is exact off $Z$, so is $(\E^\bullet)^{\otimes \l}$;
  consider the element $[h( (\E^\bullet)^{\otimes \l})]$ of $\Gr_0(C_\l\times G, Z)$ obtained from its total homology:
  \begin{equation}
 [h( (\E^\bullet)^{\otimes \l})]:=\sum\nolimits_{j} (-1)^j [{\rm H}^j((\E^\bullet)^{\otimes \l})]
  \end{equation}
  as in \cite{SouleBook}.
  By  \ref{nonsmoothlambda}, we have
  \begin{equation}\label{545}
  \al=L_\bullet (1)= [h( (\E^\bullet)^{\otimes \l})]\cdot \lambda_{-1}({\mathscr N}_{P|P^{\l}})_{|Z}^{-1}.
  \end{equation}
  We now use two results of K\"ock. By \cite[Theorem 5.1]{KockTAMS}, we have 
  canonical $C_\l\times G$-isomorphisms $ {\rm H}^j((\E^\bullet)^{\otimes \l})\simeq 
  \wedge ^j ({\mathscr N}_{Z|P}\otimes \alpha)$, where $\alpha$ again is the augmentation ideal. Therefore, 
 \begin{equation}\label{546}
  [h( (\E^\bullet)^{\otimes \l})]=\sum\nolimits_j (-1)^j [\wedge ^j ({\mathscr N}_{Z|P}\cdot \alpha)]=\lambda_{-1}({\mathscr N}_{Z|P}\cdot \alpha).
 \end{equation}
On the other hand,   \cite[Lemma 3.5]{KockRRtensor}, gives a $C_\l$-isomorphism
$$ \Omega^1_{P}\otimes \alpha\xrightarrow{\sim} {\mathscr N}_{P|P^{\l}}$$ which, as we can easily see, is also $G$-equivariant.
Therefore,
$$
 \lambda_{-1}({\mathscr N}_{P|P^{\l}})_{|Z}=\lambda_{-1}(i^*\Omega^1_{P}\cdot \alpha)
$$
in $\Kr_0(C_\l\times G, Z)$.

Now, as in \cite[Prop. 3.2]{KockRRtensor}, we see using the $G$-equivariant splitting principle (\emph{e.g.} 
\cite{KockGRR}), that if $\F$ is a $G$-equivariant locally free coherent 
sheaf over $Z$,
then 
$$
\theta^{\l}(\F)=\lambda_{-1}(\F\cdot \alpha)
$$
in $\Kr_0(C_\l\times G, Z)'/(v)\cdot 
\Kr_0(C_\l\times G, Z)'$.

Combining (\ref{545}), (\ref{546}), and the above, we obtain
\begin{eqnarray*}
\al=L_\bullet (1)= [h( (\E^\bullet)^{\otimes \l})]\cdot \lambda_{-1}({\mathscr N}_{P|P^{\l}})_{|Z}^{-1}=\ 
\ \ \ \ \ \ \ \ \ \ \\
\ \ \ \ \ \ \ \ \ =\lambda_{-1}({\mathscr N}_{Z|P}\cdot \alpha)\cdot \lambda_{-1}(i^*\Omega^1_P\cdot \alpha)^{-1}=
\theta^{\l}({\mathscr N}_{Z|P})\cdot \theta^{\l}(i^*\Omega^1_{P})^{-1}\notag
\end{eqnarray*}
and therefore
$$
\al=L_\bullet(1)=\theta^{\l}(T^\vee_Z)^{-1}
$$
in $\Gr_0(C_\l\times G, Z)^\wedge$. This   completes the proof.
\end{proof}
 \end{para}

\subsection{A general Adams-Riemann-Roch formula}\label{generalARRsec}

\begin{para}
In this section, we assume that $G$ acts tamely on $f: X\to \Spec(\Z)$ which is always projective and flat of relative dimension $d$. 
We also assume that $f$ is a local complete intersection and that $\pi: X\to Y=X/G$ is flat.
\end{para}

\begin{para} We now see that our results imply:
\begin{thm}\label{ARRtame}
Let $\F$ be a $G$-equivariant coherent locally free $\O_X$-module. 
Under the above assumptions on $X$, if $\l$ is a prime with $(\l, \#G)=1$ and $\l'$
is another prime with $\l\l'\equiv 1\mod {\rm exp}(G)$, we have
\begin{eqnarray}\label{ARRtameeq}
(\l-1)\cdot \psi^\l(\overline\chi(X, \F))=\ \ \ \ \ \ \ \ \ \ \ \ \ \ \ \ \ \ \ \ \ \ \ \  \ \ \ \ \ \ \ \ \ \ \\
\ \ \ \ \ \ \ \ \ \  =\l(\l-1)\cdot \psi^{\CNT}_{\l'} (\overline\chi(X, \F))
= (\l-1)\cdot f^\wedge_*(\theta^{\l}(T^\vee_X)^{-1}\otimes  \psi^\l(\F))\notag
\end{eqnarray}
in $\Cl(\Zl[G])^\wedge$. Here $\theta^{\l}(T^\vee_X)^{-1}$ belongs to $\Kr_0(G, X)^\wedge$.
\end{thm}

\begin{proof} In this, we also denote by $ \psi^\l$ the action 
of this operator on the completion
$\Cl(\Zl[G])^\wedge$ of the quotient $\Cl(\Zl[G])'=\Kr_0(\Zl[G])'/\langle [\Zl[G]]\rangle$; this action is well-defined  by Proposition \ref{adamsonfree} and \cite[Prop. 2.10]{KockAdamsCamb}. The statement then follows from the Adams-Riemann-Roch identity (\ref{ARRzeta}) and Theorem \ref{Bottthm}
by using also $\zeta^\wedge(\al)=\zeta^\wedge(\theta^{\l}(T^\vee_X)^{-1})=(\l-1)\cdot\theta^{\l}(T^\vee_X)^{-1}$ as in (\ref{zeta}) and Proposition \ref{KockvsCNT}.
\end{proof}
\end{para}

\quash{
\subsubsection{} Suppose $G$ is a $p$-group for an odd prime $p$, and $\l\neq p$. We claim that, in this case,  the $I_G$-adic completion
$\Cl(\Zl[G])^\wedge$ is the $p$-power part $\Cl(\Zl[G])_p$: Recall that $\Cl(\Zl[G])$
is a finite $\scrR(G)$-module. We observe that the class  $[\Zl[G]]\in \scrR(G)$ annihilates $\Cl(\Zl[G])$ but
since it has rank $\#G$, it is also invertible in the localizations 
of $\scrR(G)$ at  $\rho=I_G+(q)$, for all $q\neq p$.
Hence, the completion $\Cl(\Zl[G])^\wedge$ is supported at $p$ and the claim follows
since, by \ref{idealspara}, the only  prime ideal of $\scrR(G)$ supported 
over $p$ is $I_G+(p)$.
(cf. the proof of \cite[Proposition 4.5]{PaCubeInvent}). Therefore, 
in this case, by Proposition \ref{cft} and the above, there is an infinite set of primes $\l$,
such that $\l\mod p\in (\Z/p\Z)^*$ is a generator and $\l^{p-1}\equiv 1\mod {\rm exp}(G)$,
with
\begin{equation}
\psi^{\CNT}_\l(\overline\chi(X, \F))=\l'\cdot f^\wedge_*(\theta^{\l'}(T^\vee_X)^{-1}\otimes \psi^{\l'}(\F))
\end{equation}
in the $p$-power part $\Cl(\Z[G])_p$, provided $\l'$ is another prime with $\l \l'$ sufficiently congruent to 
$1$ modulo powers of $p$.
\end{para}}

\begin{para}\label{Cotangent} The result in this paragraph will be used  in \S \ref{curvesSec}.
Suppose in addition that $Y\to \Spec(\Z)$ is regular (then the condition that $\pi$ is   flat
is implied by the rest of our assumptions, see \emph{e.g.} \cite[(18.H)]{Matsumura}). 
Recall we have classes $T^\vee_X$, $\pi^*T^\vee_Y$ 
and $T^\vee_{X/Y}:=T^\vee_X-\pi^*T^\vee_Y$ 
in the Grothendieck group $\Kr_0(G, X)$.
Notice that $\theta^{\l}(T^\vee_Y)$, $\theta^{\l}(T^\vee_Y)^{-1}$ are defined in $\Kr_0(Y)'$
and $\theta^{\l}(T^\vee_X)^{-1}$ in $\Kr_0(G, X)^\wedge$
and we have 
$$
\theta^{\l}(T^\vee_{X/Y})^{-1}=\theta^{\l}(T^\vee_X)^{-1} \pi^*(\theta^{\l}(T^\vee_Y) )
$$
in $\Kr_0(G, X)^\wedge$. Then also
\begin{equation}\label{thetaXY}
\theta^{\l}(T^\vee_X)^{-1}=\theta^{\l}(T^\vee_{X/Y})^{-1}\pi^*(\theta^{\l}(T^\vee_Y))^{-1} 
\end{equation}
in $\Kr_0(G, X)^\wedge$.
We can now write
$$
\theta^{\l}( T^\vee_Y)^{-1}=\l^{-d}+ c_\l, \qquad \theta^{\l}( T^\vee_{X/Y})^{-1}=1+r^{\l}_{X/Y}
$$
with  $c_\l\in \Kr_0(Y)'=\Gr_0(Y)'$
supported on a 
proper closed subset of $Y$ and $r^{\l}_{X/Y}\in  \Kr_0(G, X)^\wedge$.
(Here  $1= [\O_X]$.)
From Theorem \ref{ARRtame} and (\ref{thetaXY}) we obtain
\begin{equation}
 (\l-1)\cdot \psi^\l(\overline\chi(X,\O_X))=(\l-1)\cdot f^\wedge_*(  (1+r^{\l}_{X/Y})\cdot  (\l^{-d}+\pi^* c_\l) ).
\end{equation}
 This gives the identity
\begin{eqnarray}\label{723}
\ \ \ \ (\l-1)(\psi^\l-\l^{-d})\cdot(\overline\chi(X,\O_X))=\ \ \ \ \ \ \  \ \ \ \ \ \ \ \ \ \ \ \ \ \ \ \ \ \ \ \ \ \\
\ \ \ \ \ \ \ \ \ \ \ \ \ \ \ \ \ =\l^{-d}(\l-1) f^\wedge_*(r_{X/Y}^{\l})+(\l-1)f^\wedge_*(\pi^*(c_\l)\cdot \theta^{\l}(T^\vee_{X/Y})^{-1})\notag
\end{eqnarray}
in $\Cl(\Zl[G])^\wedge$.  
 In this situation, since $X$ and $Y$ are both local complete intersections over $\Z$, the cotangent complexes
$L_{X/\Z}$, $L_{Y/\Z}$, and hence $\pi^*L_{Y/\Z}$, are all perfect of $\O_X$-tor amplitude in $[-1, 0]$.
(Here $L_{X/\Z}$, $\pi^*L_{Y/\Z}$, are complexes of $G$-equivariant $\O_X$-modules;
see \cite{IllusieCotangent}
for the definition and properties of the cotangent complex.)
If $i: Y\hookrightarrow P$ is an embedding in a smooth scheme as before, then $i$ is  a regular immersion and there  
 is a quasi-isomorphism $L_{Y/\Z}\simeq [{\mathscr N}_{Y|P}\to i^*\Omega^1_{P/\Z}]$;
similarly for $L_{X/\Z}$ after choosing a $G$-equivariant embedding.
There is a canonical distinguished triangle
$$
\pi^* L_{Y/\Z}\to L_{X/\Z}\to L_{X/Y}\to \pi^*L_{Y/\Z}[1].
$$
The cotangent complex $L_{X/Y}$ is then also perfect and gives the class 
$T^\vee_{X/Y}$ in $\Kr_0(G,X)$.
The morphism  $\pi^* L_{Y/\Z}\to L_{X/\Z}$  is an isomorphism over 
the largest open subscheme $U$
of $X$ such that $\pi: U\to V=U/G$ is \'etale. 
\end{para}

 \section{Unramified covers}

 \subsection{The main identity} 
 
 \begin{para} Here we suppose in addition that $\pi: X\to Y$ is unramified, \emph{i.e.} that the cover $\pi$ is \'etale. 
 In this case, by \'etale descent, pull-back by $\pi$ gives isomorphisms $\pi^*: \Gr_0(Y)\xrightarrow{\sim} \Gr_0(G, X)$,  $\pi^*: \Gr_0(C_\l, Y)\xrightarrow{\sim} \Gr_0(C_\l\times G, X)$. 
 We   can see that, by using such isomorphisms, the map
$$
\Delta_*: \Gr_0(C_\l, Y')\to \Gr_0(C_\l, (Y')^{\l})
$$
 can be identified with the map $\Delta_*$ of \ref{cyclicloc2}. In this case, we can 
show more directly, using localization for the $C_\l$-action on $(Y')^{\l}$, that $\Delta_*$
above
gives an isomorphism after localizing at each maximal ideal $\rho_{(\sigma, (q))}$, for $q\neq \l$,
of $\scrR(C_\l)=\Gr_0(\Zl[C_\l])'$ that contains the element $v$. Set $\scrR(C_\l)^\flat=\scrR(C_\l)/(v)$ and use ${}^\flat$ to denote 
base change via $\scrR(C_\l)\to \scrR(C_\l)^\flat$.
We see that $\Delta_*^\flat$ is an isomorphism and there is $\eta^\ell\in \Gr_0(C_\l, Y')^\flat$
whose pull back by $\pi$ maps to $\al\in \Gr_0(C_\l\times G, X')^\wedge$.
In particular, $\zeta(\eta^\ell)$ makes sense as an element in $\Gr_0(Y')$ and pulls back to an element of $\Gr_0(G, X')$ which is equal to $\zeta^\wedge(\al)$ in $\Gr_0(G, X')^\wedge$. Given the above, the proof of the identity (\ref{ARRzeta}) now 
goes through without having to take completions and we obtain
\begin{equation}\label{thmARR2}
(\l-1)\cdot \psi^\l(f^\ct_*(\O_X))=f^\ct_*(\pi^*\zeta(\eta^\ell))
\end{equation}
in $\Kr_0(\Zl[G])'$.
 \end{para}
 
 \begin{para}
 We will now show, by Noetherian induction, the following result.
 
 \begin{thm}\label{eigenthm}
 Suppose $\pi: X\to Y$ is a $G$-torsor with $Y\to \Spec(\Z)$ projective and flat of relative dimension $d$ and let  $\F$ be a $G$-equivariant coherent  
  $\O_X$-module. Let $\l$ be a prime such that $(\l, \#G)=1$.
 Then
 \begin{equation}\label{rel615}
 (\l-1)^{d+1}\cdot \prod_{i=0}^d (\psi^\l-\l^{-i})\cdot f^\ct_*(\F)=0
 \end{equation}
 in $\Cl(\Z'[G])'=\Cl(\Zl[G])[\l^{-1}]$.
 \end{thm}
 
 \begin{proof}  Set
 $\Psi(\l, d)= (\l-1)^{d+1}\cdot \prod_{i=0}^d (\psi^\l-\l^{-i}) $  for the endomorphism
 of  $\Cl(\Z'[G])'=\Kr_0(\Zl[G])'/\langle [\Zl[G]]\rangle$;
this is well-defined by Proposition \ref{adamsonfree}. We start by recalling that 
since $\pi: X\to Y$ is a $G$-torsor, by descent, all $G$-equivariant coherent $\O_X$-modules $\F$   are obtained by pulling back along $\pi$, \emph{i.e.} are of the form $\F\simeq \pi^*\GG$, for $\GG$ a coherent $\O_Y$-module. By \cite[Theorem 6.1]{CEPTRR}, the image of the class 
$f^\ct_*(\pi^*\GG)$ in $\Cl(\Zl[G])_\rho$ is trivial for  any prime $\rho$ of $\scrR(G)$
that does not contain the ideal $I_G$ (see the proof of Prop. 4.5 in \cite{PaCubeInvent}). Therefore, it is enough to 
consider   (\ref{rel615}) for the image of $f^\ct_*(\pi^*\GG)$  in $\Cl(\Zl[G])^\wedge$.
We will argue by induction on $d$. 
 \quash{ Notice that $(\l-1)\cdot \psi^\l(x_0)=\zeta(\tau^{\l}(x_0))=(\l-1)\cdot x_0$, for $x_0=[\Z'[G]]$:
Indeed, the first identity follows from Proposition \ref{Adams} and (\ref{zeta}).
The second follows after observing that $\Zl[G]^{\otimes \l}=\Zl[G^\l]\simeq \Z'[G]\oplus {\rm Ind}^{C_\l\times G}_G(F)$, for $F$ a free $\Zl[G]$-module, where the factor $\Zl[G]$ has trivial 
$C_\l$-action (\emph{cf.} proof of Theorem 4.9 in
\cite{KockRRtensor}). Hence,   $\Psi(\l, d)
 = (\l-1)^{d+1}\cdot \prod_{i=0}^d (\psi^\l-\l^{-i}) $  
 gives an endomorphism of   $\Cl(\Z'[G])'$.}
The map $\GG\mapsto f^\ct_*(\pi^*\GG)$ induces a group homomorphism 
 $$
 \chi: \Gr_0(Y)\to \Cl(\Z'[G])'.
 $$
Now note that   $\Gr_0(Y)$ is generated by classes of the form $i_*(\O_T)$
 where $i: T\hookrightarrow Y$ is an integral closed subscheme of $Y$. 
Let us consider 
the $G$-torsor $\pi_{|T}: \pi^{-1}(T)=X\times_YT\to T$ obtained by restriction and denote  by
$h: \pi^{-1}(T)\to \Spec(\Z)$ the structure morphism.  Observe that by the definitions, we have $f^\ct_*\cdot\pi^*\cdot  i_*=h^\ct_*\cdot \pi^*_{|T}$. 
 If $T$ is fibral, \emph{i.e.} if $T\to \Spec(\Z)$ factors through 
  $\Spec({\mathbb F}_p)$ for some prime $p$, then $\chi(i_*(\O_T))=0$
 by a result of Nakajima \cite{NakajimaInv}, see \cite[Theorem 1.3.2]{CEPTAnnals}. It remains to deal with $T$ which are integral and 
 flat over $\Spec(\Z)$;  the above   allows us to reduce to the case $Y$ is integral and $\GG=\O_Y$. 
We first show that for $Y$ integral,
projective and flat over $\Spec(\Z)$ of dimension $d+1$, there is a proper reduced closed subscheme $i: W\hookrightarrow Y$
(therefore of smaller dimension) and a class $c'\in \Gr_0(C_\l, W')^\flat$ such that 
\begin{equation}\label{eq616}
\eta^\ell-\l^{-d}=i_*(c').
\end{equation}
To see this take $W$ to be given by the complement $Y-U$ of an open subset $U$ of $Y$ where
$(\Omega^1_{Y/\Z})_{|U}\simeq\O_U^d$.  Indeed,  then since $U$ is smooth, by the easy case
of Theorem \ref{Bottthm} (\cite{KockRRtensor}), the element $\eta^\ell_{U'}$ for $U'$ is $\theta^{\l}(\O_{U'}^d)^{-1}=\l^{-d}$.
Since $\Delta_*$ and therefore $L_\bullet$ commutes with restriction to open subschemes,
we obtain that the restriction of $\eta^\ell$ to $U$ is $\l^{-d}$ and so there is $c'$ as above. Applying
$\zeta$ to (\ref{eq616}) we obtain
\begin{equation}
\zeta(\eta^\ell)-(\l-1)\l^{-d}=i_*(\zeta(c'))
\end{equation}
with $\zeta(c')\in \Gr_0(W')'$. The identity (\ref{thmARR2}) now implies
\begin{equation}\label{618}
(\l-1)\cdot (\psi^\l-\l^{-d})(f^\ct_*(\O_X))= f^{\rm ct}_*(\pi^*i_*(\zeta(c')))
\end{equation}
in $\Cl(\Zl[G])^\wedge$.
Consider 
the $G$-torsor $\pi_{|W}: \pi^{-1}(W)=X\times_YW\to W$ obtained by restriction and denote  by
$h: \pi^{-1}(W)\to \Spec(\Z)$ the structure morphism.  Again, we have $f^\ct_*\cdot\pi^*\cdot  i_*=h^\ct_*\cdot \pi^*_{|W}$. We can 
extend $\zeta(c')\in \Gr_0(W')'$ to a class $z\in \Gr_0(W)'$.
Then
$$
 f^\ct_*(\pi^*i_*(\zeta(c')))=h^\ct_*(\pi_{|W}^*z).
$$
This identity allows us to reduce to considering the  $G$-torsor $\pi_{|W}$.
If $d=0$, then $W$ is fibral and $h^\ct_*(\pi_{|W}^*z)=0$
by the result of Nakajima (in this case, the normal basis theorem 
and devissage is enough).
For $d>0$, we can use the induction hypothesis on $h$ 
and the $G$-torsor $\pi^{-1}(W)\to W$. This implies that $\Psi(\l, d-1)$
annihilates $h^\ct_*(\pi_{|W}^* z)$ and the result follows from (\ref{618})
and the above.
\end{proof}

\begin{Remark}\label{useTaylor}
{\rm  M. Taylor's  proof of the Fr\"ohlich conjecture \cite{MJTFrohlichConj} 
easily implies that $2\cdot f_*^{\ct}(\F)=0$ in $\Cl(\Z[G])$
if $d=0$ (see \cite[Theorem 4.1]{PaCubeInvent}). By using this input at the step $d=0$ 
of the inductive proof above, we can improve Theorem \ref{eigenthm} 
and obtain the following: If $d\geq 1$ and $\l$ is odd, then $f_*^\ct(\F)$ is actually annihilated
by $(\l-1)^{d+1}\cdot \prod_{i=1}^d (\psi^\l-\l^{-i})$ in $\Cl(\Zl[G])'$, \emph{i.e.} we can omit the factor $\psi^\l-1$ 
that corresponds to $i=0$.}
\end{Remark}
 \end{para}

  \subsection{Class groups of $p$-groups and class field theory}
  
  In this section, $G$ is a $p$-group  of exponent $p^N$ with $p$ an odd prime.

 \begin{para} 
 By \cite{Roquette}, we can write the group algebra $\Q[G]$
  as a direct product
 of matrix rings
 $ 
 \Q[G]=\prod_i {\rm Mat}_{n_i\times n_i}(\Q(\zeta_{p^{s_i}}))
 $ 
 with center 
 $ 
 Z=\prod_i K_i.
 $ 
 Here $K_i=\Q(\zeta_{p^{s_i}})$ is the cyclotomic field, $s_i\leq N$. Denote by $\O_i$ the ring of integers of $K_i$.
\end{para}

 \begin{para}
Denote by $\Cl(\Z[G])_p$, $\Cl(\Zl[G])_p$, the $p$-power torsion 
parts of $\Cl(\Z[G])$, $\Cl(\Zl[G])$. We now show:
 
 \begin{prop}\label{cft}
There exists an infinite set of primes $\l\neq p$ with the following properties:

a) $\l\mod p$ generates  $(\Z/p\Z)^*$,

b)  $\l^{p-1}\equiv 1\mod p^N$ and,

c) for $\Zl=\Z[\l^{-1}]$ the restriction
 $
 \Cl(\Z[G])_p\to \Cl(\Zl[G])_p
 $
 is an isomorphism. 
 \end{prop}
 
 \begin{proof}  In what follows, we will write ${\Bbb A}(K)$, resp.
${\Bbb A}(K)^*$, for the adeles, resp. ideles, of the number field $K$
and ${\Bbb A}(\O_K)^*=\prod_{v} \O_{K, v}^*$ for the integral
ideles.
Using the Fr\"ohlich description of the class group \cite{MJTclassgroups}, we can write 
\begin{equation}\label{Fro}
\Cl(\Z[G])=\left({\Bbb A}(Z)^*/Z^*\right )/\UU=\left(\prod_i {\Bbb A}(K_{i})^* /K_{i}^*\right)/ \UU
\end{equation}
where $\UU=\prod_v {\rm Det}(\Z_v[G])\subset  {\Bbb A} (\O_{Z})^*=\prod_i  {\Bbb A}(\O_{i})^*$.
The subgroup $\UU$ is open, of finite index in ${\Bbb A} (\O_{Z})^*$. We can find finite index 
subgroups ${\UU}_i=\prod_{v} \UU_{i, v}\subset  {\Bbb A}(\O_{i})^*$,  
with  $\UU_{i, v}=(\O_{i})^*_v$ if $v$ does not divide $p$,
such that 
\begin{equation}\label{subg}
\prod_i  \UU_i\subset \UU\subset {\Bbb A} (\O_Z)^*.
\end{equation}
We can also assume that $\UU_{i,p}$ is stable for the action of the Galois group
$\Gal(K_{i}/\Q)=\Gal(K_{i, p}/\Q_p)$. 
It follows now from (\ref{Fro}) and (\ref{subg})  that we can write $\Cl(\Z[G]) $ as a quotient
\begin{equation}\label{quo}
  \prod_i {\Bbb A}(K_{i})^* /(K_{i}^* \cdot\UU_i)\twoheadrightarrow \Cl(\Z[G]).
\end{equation}
By class field theory, the source can be identified with the 
product of the   Galois 
groups of  ray class field extensions of $K_{i}$ which are 
at most ramified at 
the unique prime over $p$. In fact, this also implies that the $p$-Sylow
${\rm Cl}(\Z[G])_p$ can also be written
as a quotient  
$$
\prod_i \Gal(L_{i}/K_{i})\twoheadrightarrow \Cl(\Z[G])_p 
$$
where $L_{i}/K_{i}$ is a $p$-power ray class field
of $K_{i}$ which is ramified at most at the unique prime 
of $K_{i}$ over $p$ such that $\Gal(L_{i}/K_{i})$ is identified with the $p$-power quotient
of ${\Bbb A}(K_{i})^* /(K_{i}^* \cdot\UU_i)$.
 Let us consider $\Cl(\Zl[G])$ for $\l\neq p$. The discussion above applies again and as before,
we can write 
$$
\prod_i {\Bbb A}^{\l}(K_{i})^* /(K_{i}^* \cdot\UU_i^\l)\twoheadrightarrow \Cl(\Zl[G]) 
$$
where the superscript $\l$, as in ${\Bbb A}^{\l}$,  means adeles away from $\l$. 
(Also, as usual, ${\Bbb A}_\l$ will denote adeles over $\l$.) We can easily see that
$\Cl(\Z[G])\to \Cl(\Zl[G])$ and hence $
 \Cl(\Z[G])_p\to \Cl(\Zl[G])_p
 $ is surjective.

Now choose $n$ such that $n\geq N\geq {\rm max}_i\{s_i\}$ and 
such that $\Q(\zeta_{p^n})$ contains all the (cyclotomic) $p$-power ray class fields $L_{i}$, for $i$ with $s_i=0$,
of $K_i=\Q$.  
Also set $M_n$ to be a $p$-ray class field of $\Q(\zeta_{p^n})$ ramified only above $p$
that contains all the fields $L_{i}$, for $s_i\geq 1$, and corresponds to a $
{\rm Gal}(\Q(\zeta_{p^n})/\Q)$-stable subgroup of ${\mathbb A}(\Q(\zeta_{p^n}))^*$.
 The extension $M_n/\Q$  is Galois; set $G_n={\rm Gal}(M_n/\Q)$. This is an extension
 \begin{equation}
 1\to {\rm Gal}(M_n/\Q(\zeta_{p^n}))\to G_n\to {\rm Gal}(\Q(\zeta_{p^n})/\Q)=\Delta\times \Gamma_n\to 1
 \end{equation}
 with $\Delta={\Gal}(\Q(\zeta_p)/\Q)\simeq  (\Z/p\Z)^*$, $\Gamma_n=\Gal(\Q(\zeta_{p^n})/\Q(\zeta_p))\simeq (\Z/p^{n-1}\Z)$.
 We can find an element $\tau\in G_n$ of order $p-1$
as follows: lift the generator of $\Delta$ to an element $\tau_0$, then a suitable power $\tau=\tau_0^{p^m}$
has order $p-1$.  
By the \v Chebotarev density theorem, there is an
infinite set of primes $\l$ so that
${\rm Frob}_\l$ lies in the conjugacy class $\langle \tau\rangle$ of $\tau$ in $G_n$. Then ${\rm Frob}_\l$
generates the group $\Gal(\Q(\zeta_p)/\Q)$. Hence, the ideal $(\l)$ 
remains prime in $\Gal(\Q(\zeta_p)/\Q)$
and  $\l$ generates $(\Z/p\Z)^*$. 
Write ${\mathfrak L}=(\l)$ in $\Q(\zeta_p)$
and consider ${\rm Frob}_{\mathfrak L}$ in $\Gal(M_n/\Q(\zeta_p))$.
We have ${\rm Frob}_{\mathfrak L}={\rm Frob}^{p-1}_\l=\langle \tau^{p-1}\rangle=
\langle 1\rangle$ and so ${\mathfrak L}$ splits completely in $M_n$. 
 
 By class field theory and the assumption $\UU_{i, v}=\O_{i, v}^*$
 for $v\neq p$, we see that the above give:
 There is an infinite set of primes $\l\neq p$
 that generate $(\Z/p\Z)^*$, satisfy $\l^{p-1}\equiv 1\mod p^N$,
 and in addition: For all $i$, 
the image of the subgroup 
$$
{\Bbb A}_l(K_{i})^*=\prod_{{\mathfrak Q}|l}  K_{i, {\mathfrak Q}}^*\subset 
 {\Bbb A}(K_{i})^*
$$
in ${\Bbb A}(K_{i})^*/K_{i}^* \cdot\UU_i$ has order prime to $p$. Let us denote
by $Q_i=Q_i(\l)$ this order and set  $Q=\prod_i Q_i$.
For such an $\l$, suppose $c$ is an
element in the kernel of $\Cl(\Z[G])_p\to \Cl(\Zl[G])_p$
which is given by an idele $(a_i)\in \prod_i {\Bbb A}(K_{i})^*$.
Then  
$$
(a_i^\l)_i=(\gamma_i)_i\cdot u^{\l}
$$
with $\gamma_i\in K_{i}^*$ (diagonally embedded in the
prime to $\l$-ideles) and ${u}^{\l}\in \UU^{\l}$. Here
$
(a_i)_i=(a_i^\l)_i\cdot (a_{i,\l})_i
$
with $(a_{i,\l})$ the $\l$-component of $(a_i)$ (considered as an idele with
$1$ at all places away from $\l$).
The product idele
$
(b_i)_i=(\gamma_i^{-1}\cdot a_i)_i
$
also produces the class $c$ in $\Cl(\Z[G])_p$.
We can  write
$
(b_i)_i=(b_i^\l)_i\cdot (b_{i,\l})_i.
$
The component $(b_{i,v})_i$ at a place $v$ away from $\l$ is
equal to the corresponding component of $u^{\l}$ and so it is in $\UU_v$.
 Using our assumption on $\l$  
we can write
$$
(b_{i,\l}^Q) =(\delta_i\cdot u_i)
 $$
with $u_i\in \UU_i$, $\delta_i\in  K^*_{i}\subset {\Bbb A}(K_{i})^* $
(embedded diagonally). Combining these gives
$$
(b_i)_i^{Q}=(\delta_i\cdot u_i)_i\cdot (b_i^\l)^{Q}_i
$$
which is in $(\prod_i K^*_{i})\cdot \UU$.
Therefore $Q\cdot c$ is trivial in $\Cl(\Z[G])$,
hence $c$ is trivial.
\end{proof}
\end{para}

  \subsection{Adams eigenspaces and the proof of the main result}

 \begin{para} Recall $G$ is  a $p$-group of exponent $p^N$
for an odd prime $p$.
 The group $(\Z/p^N\Z)^*=(\Z/p\Z)^*\times \Z/p^{N-1}\Z$
 acts on the $p$-power torsion $\Cl(\Z[G])_p$ via the Cassou-Nogu\`es--Taylor Adams operations:
Indeed, these operators are periodic and this is essential for our argument. The element
 $a\in (\Z/p^N\Z)^*$ acts via $\psi^{\rm CNT}_a$; we will simply denote this by $\psi_a$
 in what follows. We will restrict this action to the subgroup $(\Z/p\Z)^*$. This gives a decomposition
 into eigenspaces
 $$
 \Cl(\Z[G])_p=\bigoplus_{\chi: (\Z/p\Z)^*\to \Z_p^*}\Cl(\Z[G])_p^{\chi}=
\bigoplus_{i=0}^{p-2} \Cl(\Z[G])_p^{(i)}
 $$ 
where
$
 \Cl(\Z[G])_p^{(i)}=\{c\in  \Cl(\Z[G])_p\ |\ \psi_a(c)=\omega(a)^i\cdot c\,, \forall a\in (\Z/p\Z)^*\}.
 $
Here $\omega: (\Z/p\Z)^*\to \Z_p^*$ is the
 Teichm\"uller character. Similarly for $\Cl(\Zl[G])_p$
 after we choose a  prime $\l\neq p$.
 
 \begin{thm}\label{eigencor}
  Suppose $\pi: X\to Y$ is a $G$-torsor with $Y\to \Spec(\Z)$ projective and flat of relative dimension $d$ and assume that $G$ is a $p$-group  for an odd prime $p$.  
 Let $\F$ be a $G$-equivariant coherent   $\O_X$-module. Then the class $\overline\chi(X, \F) \in \Cl(\Z[G])$ is $p$-power torsion 
and if $p>d$ lies in 
 $
\bigoplus_{i=2}^{d+1} \Cl(\Z[G])_p^{(i)}.
 $
 In particular, if $d=1$, then $\overline\chi(X, \F)$ lies in the eigenspace $\Cl(\Z[G])_p^{(2)}$.
 \end{thm}
 
 \begin{proof}
The fact that $\overline\chi(X, \F)$ is $p$-power torsion follows from 
the localization theorem of \cite{CEPTRR} as in \cite[Prop. 4.5]{PaCubeInvent}.   
Choose an odd prime $\l$ as in Proposition \ref{cft}; in particular $\l-1$ is prime to $p$.
Propositions \ref{adamsonfree} and Proposition \ref{KockvsCNT} together imply
that, for  $x\in \Cl(\Zl[G])_p$, we have $  \psi^\l(x)=\l \cdot \psi_{\l'}(x)$
 for $\l' \l\equiv 1\mod p^N$.
From Theorem \ref{eigenthm} and Remark \ref{useTaylor}, we 
then obtain that $f^\ct_*(\F)$ lies in $\oplus_{i=2}^{d+1}\Cl(\Zl[G])^{(i)}_p$.
However, for our choice of $\l$, the restriction $\Cl(\Z[G])_p\to \Cl(\Zl[G])_p$
is an isomorphism and this gives the result.
 \end{proof}
 \end{para}
  
  \begin{para}
We continue to assume that
$G$ is a $p$-group, $p$ an odd prime. Let $\#G=p^n$.
We show:

\begin{prop} \label{bernoulli}  
If $p\geq 5$, then $\Cl(\Z[G])_p^{(2)}=(0)$. If $p=3$, then $\Cl(\Z[G])_p^{(2)}=
\Cl(\Z[G])_p^{(0)}$ is annihilated by the Artin exponent $A(G)=p^{n-1}$ of $G$.
\end{prop}

\begin{proof}  As above, we can write
$$
\Cl(\Z[G])={\prod_i {\Bbb A}(K_{i})^* }/{(\prod_i K_{i}^*) \cdot \UU},
$$
where $K_{i}=\Q(\zeta_{p^{s_i}})$ and $\UU$ is an open subgroup of the product $\prod_i {\Bbb A}(\O_{i})^* $
which is maximal at all $v\neq p$.
For $b\in (\Z/p\Z)^*$, denote by $\sigma_b$ the Galois automorphism of $\Q(\zeta_{p^\infty})=\cup_m \Q(\zeta_{p^m})$
given by $\sigma(\zeta)=\zeta^{\omega(b)}$. We can see (\cite{CNTAdams})
that the operator $\psi_a$, for $a\in (\Z/p\Z)^*\subset (\Z/p^N\Z)^*$, is induced by the action of $\sigma_{ a}$ on the product $\prod_i {\Bbb A}(K_{i})^* $.
Let $\M_G\simeq \oplus_i {\rm Mat}_{n_i\times n_i}(\Z[\zeta_{p^{s_i}}])$ be a maximal $\Z[G]$-order
in $\Q[G]$. Denote by $D(\Z[G])$ the kernel of the natural group homomorphism
$
\Cl(\Z[G])\to \Cl(\M_G)=\prod\nolimits_i \Cl(\Q(\zeta_{p^{s_i}})).
$
The kernel  group $D(\Z[G])$ has $p$-power order (\cite[p. 37]{MJTclassgroups}).
Since $\UU$ is maximal at  $v\neq p$ we can write
$$
D(\Z[G])=\frac{\prod_i  \Z_p[\zeta_{p^{s_i}}]^*}{(\prod_i \Z[\zeta_{p^{s_i}}]^*)\cdot \UU_p}.
$$
For $x\in \Cl(\Z[G])_p^{(2)}$, let $(x_i)_i$ be the image of $x$ in the class group $\Cl(\M_G)$.
Then $x_i$ is a $p$-power torsion element in $\Cl(\Q(\zeta_{p^{s_i}}))$
which satisfies $\sigma_a(x_i)=\omega(a)^2\cdot x_i$, for all $a\in (\Z/p\Z)^*$.  The second eigenspace
of the $p$-part of the class group of $\Q(\zeta_{p^m})$ is trivial. (Combine $ B_2=1/6$ with Herbrand's
theorem and the ``reflection theorems", see \cite[Theorems 6.17,  10.9]{Washington}, to see this for $m=1$; the result then follows.)
It follows that $(x_i)_i=0$ and so $x$ is in $D(\Z[G])$. Such an $x$ is then represented
by $(u_i)_i$ with $u_i\in    (\Z_p[\zeta_{p^{s_i}}]^*)^{(2)}$. For $m\geq 0$, consider the pro-$p$-Sylow subgroup
$ (\Z_p[\zeta_{p^{m}}]^*)_p$ of $\Z_p[\zeta_{p^{m}}]^*$. Denote by $(\widehat{\Z[\zeta_{p^m}]^*})_p$ the 
intersection $\widehat{\Z[\zeta_{p^m}]^*}\cap (\Z_p[\zeta_{p^{m}}]^*)_p$ of the $p$-adic 
closure of the global units $ \Z[\zeta_{p^m}]^*$ in $\Z_p[\zeta_{p^{m}}]^*$
with   $ (\Z_p[\zeta_{p^{m}}]^*)_p$. If $p\geq 5$, then, since  the second Bernoulli number $B_2=1/6$ is
 not divisible by $p$, we have
$$
(\Z_p[\zeta_{p^{m}}]^*)_p^{(2)}=(\widehat{\Z[\zeta_{p^m}]^*})_p^{(2)}
$$
by a classical result of Iwasawa (see for example \cite[Theorem 13.56]{Washington}, \emph{cf.} \cite[p. 296]{OliverD}).
This shows that $x$ is trivial in $D(\Z[G])$. If $p=3$, then 
since $3$ is regular, we  have as above
$\Cl(\Z[G])_p^{(2)}=
\Cl(\Z[G])_p^{(0)}=D(\Z[G])_p^{(0)}$. This group is annihilated by $A (G)$
by \cite[Theorem 9]{OliverD};  in this case, $A (G)=p^{n-1}=3^{n-1}$
by \cite{Lam}.
 \end{proof}
\end{para}
 
 \begin{para}
We can now   show  Theorems \ref{main1} and \ref{mainDiffthm} of the introduction.
For this we   allow $G$ to stand for an arbitrary finite group.
\smallskip

\noindent {\it Proof of Theorem \ref{main1}.} Using  Noetherian induction and the $0$-dimensional result of Taylor exactly as
in \cite[Prop. 4.4]{PaCubeInvent}, we see that 
\begin{equation}\label{Gtorsion}
{\rm gcd}(2, \#G)\cdot \#G\cdot f^\ct_*(\F)=0.
\end{equation}
Using localization as in \cite[Prop. 4.5]{PaCubeInvent}, 
we see that the $p$-power torsion part 
of the class $f^\ct_*(\F)$ is annihilated by any power of $p$ that 
annihilates its restriction ${\rm Res}_{G_p}^G( f^\ct_*(\pi^*\GG))$
in the class group $\Cl(\Z[G_p])$ of a $p$-Sylow $G_p$. By definition, this restriction is the
Euler characteristic class for the $G_p$-cover $X\to X/G_p$.
If $G$ is a $p$-group
of order $p\geq 5$, we have $f^\ct_*(\F)=0$ in $\Cl(\Z[G])$  by Theorem \ref{eigencor} and Proposition \ref{bernoulli}.  We can apply this to a $p$-Sylow $G_p$ of $G$ and the $G_p$-torsor $X\to X/G_p$. By the above, we obtain that the prime to $6$ part of 
 $f^\ct_*(\F)$ is trivial. By (\ref{Gtorsion}) the $2$-part is always annihilated
by ${\rm gcd}(2, \#G)^{v_2(\#G)+1}$. When the $2$-Sylow $G_2$ of $G$
is abelian, \cite[Theorem 1.1]{PaCubeInvent},  applied
to  the cover $X\to X/G_2$, shows that the restriction of $f^\ct_*(\F)$ to $\Cl(\Z[G_2])$ is $2$-torsion. In general,  by \cite[Theorem 1.1]{PaCubeInvent}, the restriction of 
$f^\ct_*(\F)$ to $\Cl(\Z[G_2])$ lies in the kernel subgroup $D(\Z[G_2])$.
The kernel subgroup $D(\Z[G_2])$ is trivial when the $2$-group $G_2$ has order $\leq 4$,
is cyclic of order $8$, or is dihedral;
it has order $2$ when $G_2$ is generalized quaternion or semi-dihedral
 (\cite[II, 50.16]{CurtisReiner}, \cite[Theorem 2.1, p. 79]{MJTclassgroups}).
Also, by Theorem \ref{eigencor} and Proposition \ref{bernoulli}, when $p=3$, the 
restriction of  $f^\ct_*(\F)$ to $\Cl(\Z[G_3])$ is annihilated by the Artin exponent $A(G_3)$.
If $G_3$ is abelian, this restriction is trivial, again by \cite[Theorem 1.1]{PaCubeInvent}.
Theorem \ref{main1} now follows.

\smallskip

\noindent {\it Proof of Theorem \ref{mainDiffthm}.} 
Let us note that in this situation, we have ${\rm H}^0(X_\Q, \O_{X_\Q})={\rm H}^0(Y_\Q, \O_{Y_\Q})$, and,
since the $G$-cover $X_\Q\to Y_\Q$ is unramified, the Hurwitz formula gives $g_X-h=\#G\cdot (g_Y-h)$.
The result then follows from Theorem \ref{main1} exactly as in the proof of   \cite[Theorem 5.5]{
PaCubeInvent} provided we show that ${\rm H}^1(X, \omega_X)\simeq \Z^{\oplus h}$.
Since $G$ acts trivially on ${\rm H}^0(X_\Q, \O_{X_\Q})$, we have ${\rm H}^0(X,\O_X)\simeq \Z^{\oplus h}$ with trivial $G$-action.
Under the rest of our assumptions, duality implies that there is a $G$-equivariant isomorphism 
${\rm H}^1(X, \omega_X)\simeq{\rm Hom}_{\Z}({\rm H}^0(X, \O_X),\Z)$, and the result 
then follows. 
\endproof
\end{para}

\section{Tamely ramified covers of curves}\label{curvesSec}

\subsection{Curves over $\Z$}
We assume that $G$ acts tamely on $f: X\to \Spec(\Z)$ which is  projective,  flat and a local complete intersection
of relative dimension $1$. We also assume that $Y=X/G$ is irreducible and regular; then $\pi: X\to Y$ is finite and flat.
\quash{Let $U$ be the largest open subscheme of $X$ such that $\pi: U\to V=U/G$ is \'etale. The complements $R(X/Y)=X-U$ and $B(X/Y)=Y-V$ are respectively 
the ramification and branch locus of $\pi$. The ramification locus is the
closed subset of $X$ defined by the annihilator ${\rm Ann}(\Omega^1_{X/Y})$. Our assumption of tameness implies that 
both the ramification and branch loci are fibral, \emph{i.e.} are subsets of the union
of  fibers of $X\to \Spec(\Z)$, resp. $Y\to \Spec(\Z)$, over a finite set $S$ of primes $(p)$.  
(See \cite[1.2]{CEPTAnnals}.)
We also assume that $f$ is a local complete intersection and that $\pi: X\to Y=X/G$ is flat.
We continue to assume that $X\to Y$ is as in the beginning of \ref{generalARRsec}. 
Also, we assume in addition that $X$ has relative dimension $d=1$, that $Y$ is regular,
and that $\l\neq 2$.}
Let $U$ be the largest open subscheme of $X$ such that $\pi: U\to V=U/G$ is \'etale. The complements $R(X/Y)=X-U$ and $B(X/Y)=Y-V$ are respectively 
the ramification and branch locus of $\pi$. The ramification locus is the
closed subset of $X$ defined by the annihilator ${\rm Ann}(\Omega^1_{X/Y})$. Our assumption of tameness implies that 
both the ramification and branch loci are fibral, \emph{i.e.} are subsets of the union
of  fibers of $X\to \Spec(\Z)$, resp. $Y\to \Spec(\Z)$, over a finite set $S$ of primes $(p)$.  
(See \cite[1.2]{CEPTAnnals}.)  Fix a prime $\l\neq 2$ that does not divide $\#G$.

\begin{para}  Denote by $F_i\Gr_0(Y)$ the subgroup of elements of $\Kr_0(Y)=\Gr_0(Y)$ 
represented as linear combinations of coherent sheaves supported on subschemes
of $Y$ of dimension $\leq i$. 
Consider the homomorphisms 
$$
\overline\chi: F_1\Gr_0(Y)\to \Cl(\Z[G]);\quad \overline\chi(c)=  f^{\ct}_*(\pi^*(c)), \ \ \hbox{\rm and,}
$$
$$
\cl_{X/Y}: F_1\Gr_0(Y)\to \Cl(\Zl[G])^\wedge;\quad \cl_{X/Y}(c):=f_*^\wedge(\pi^*(c)\cdot \theta^{\l}(T^\vee_{X/Y})^{-1})
$$
where $\theta^{\l}(T^\vee_{X/Y})^{-1}$ is as in \ref{Cotangent}. Note that a value of $\cl_{X/Y}$ appears in the right hand side of (\ref{723}).

\begin{prop}\label{vanishF1}
Under the above assumptions, 

1) the image of $\overline \chi$ is ${\rm gcd}(2, \#G)$-torsion,

2) the image of   $\cl_{X/Y}$ is $(\l-1)$-torsion.
\end{prop}

\begin{proof} 
We note that  under our assumptions, we have isomorphisms ${\rm Pic}(Y)={\rm CH}_1(Y)\xrightarrow{\sim} F_1\Gr_0(Y)/F_0\Gr_0(Y)$
and ${\rm CH}_0(Y)\xrightarrow{\sim} F_0\Gr_0(Y)$.
(This follows from ``Riemann-Roch without denominators" as in \cite{SouleCanadian}, see 
also \cite[Ex. 15.3.6]{FultonInter}).  Here ${\rm CH}_i(Y)$ is the 
Chow group of  dimension $i$ cycles modulo rational equivalence on $Y$ and
both maps are given by sending the class $[V]$ of a  dimension $1$, resp. $0$,
integral subscheme $V$ of $Y$ 
to $i_*([\O_V])$ where $i: V\to Y$ is the corresponding morphism.

Now observe:

a) Both $\cl_{X/Y}$, $\overline\chi$ are trivial on $F_0\Gr_0(Y)$: 
By two-dimensional class field theory \cite[Theorem 2]{KatoSaitoAnnals}, ${\rm CH}_0(Y)$ is a finite abelian group 
and there is a reciprocity isomorphism ${\rm CH}_0(Y)\xrightarrow{\sim} \tilde\pi_1^{\rm ab}(Y)$,
where $\tilde\pi_1^{\rm ab}(Y)$ classifies unramified abelian covers of $Y$ that split completely over all real-valued
 points  of $Y$. Suppose $Y'\to Y$ is an irreducible unramified abelian Galois cover. A standard argument using the classical description of unramified abelian covers of curves via
 isogenies of their Jacobians (or, alternatively, smooth base change for \'etale cohomology), shows that there is an infinite set of primes $q$ such that the base change
 $Y'_{{\mathbb F}_q}\to Y_{{\mathbb F}_q}$ is non-split, \emph{i.e.} such that $Y'_{{\mathbb F}_q}$ is irreducible. 
 By applying this to the universal cover $Y^{\rm uni}\to Y$ with Galois group ${\rm CH}_0(Y)$, we obtain that there is a prime $q$ not in $\{\l\}\cup S$ 
 such that $Y_{{\mathbb F}_q}$ is smooth and with the property that $Y^{\rm uni}_{{\mathbb F}_q}$ is irreducible. 
 This implies that the Frobenius elements of the closed points of the smooth projective curve $Y_{{\mathbb F}_q}$ generate the Galois group, 
in other words, that the group ${\rm CH}_0(Y)$ is generated by the classes of points that are supported on $Y_{{\mathbb F}_q}\subset Y$.  
Now $\theta^{\l}(T^\vee_{X/Y})^{-1}_{|U}=1$ and since $ X_{{\mathbb F}_q} \subset U$, if $c$ corresponds to   a point on $Y_{{\mathbb F}_q}$,
we obtain $\cl_{X/Y}(c)=f_*^{\rm ct}(\pi^*(c))=0$
by the normal basis theorem for the $G$-Galois algebra that corresponds to $\pi^{-1}(c)$.
(See also \cite[Theorem 1.3.2]{CEPTAnnals}).

b) By (a) above, $\cl_{X/Y}$ and $\overline\chi$ both factor through $F_1\Gr_0(Y)/F_0\Gr_0(Y)\simeq {\rm Pic}(Y)$.
Suppose now that $\delta\in \Pic(Y)$.
By \cite[Prop. 9.1.3]{CEPTAnnals} (the assumption that the special fibers are divisors
with normal crossings is not needed for this), there is a ``harmless" base extension given
by a number field $N/\Q$, unramified at all primes over $S$, of degree $[N:\Q]$ a power of a prime number $\neq \l$,
and $[N:\Q]\equiv 1\mod \#\Cl(\Z[G])$ such that the following is true: We can write the base change 
 $\delta_{\O_N}\in {\rm Pic}(Y_{\O_N})$ 
as a sum $\delta_{\O_N}=\sum_i m_i [D_i]$
with $m_i=\pm 1$, where $D_i$ are horizontal divisors in $Y_{\O_N}$, which at most
intersect  each irreducible component  of $(Y_{\O_N})_p$, $p\in S$, 
transversely at closed points that
are away from the singular locus of the reduced special fiber  
$(Y_{\O_N})_p^{\rm red}$. Denote 
 by $\iota_i: D_i\hookrightarrow Y_{\O_N}$ the closed immersion
and by $\mathcal {D}_i$ the normalization of $D_i$. Then, we can see as in \emph{loc. cit.} that for each $i$, the morphism $\ti \DD_i=\pi^{-1}(\DD_i)=\DD_i\times_Y X\to \DD_i$
is a tame $G$-cover of regular affine schemes of dimension $1$   flat over $\Z$,
which is unramified away from $S$. The normalization morphism $q_i: \DD_i\to D_i$
is an isomorphism over an open subset of $D_i$ that contains all primes over $S$. 
As in \cite{CEPTAnnals}, we can now see that our conditions 
on the field $N$ together with  $\delta_{\O_N}=\sum_i m_i[D_i]$ 
imply that  
$$
\overline\chi (\delta)=\sum\nolimits_i m_i\cdot \overline\chi(X_{\O_N}, (\iota_i)_*\O_{D_i} )=\sum\nolimits_i m_i\cdot [\Gamma( \ti \DD_i, \O_{\ti \DD_i})]
$$
in $\Cl(\Z[G])$. Observe that the classes $[\Gamma( \ti \DD_i, \O_{\ti \DD_i})]
$
are ${\rm gcd}(2, \#G)$-torsion by Taylor's theorem and part (1) follows.

The proof of part (2) is similar: For simplicity, 
write $D=D_i$, $\DD=\DD_i$, $\ti \DD=\ti \DD_i$, $\iota: D\hookrightarrow Y_{\O_N}$, and 
denote by $ h: \ti \DD\to \Spec(\Z)$ the structure morphism. 
By Taylor's theorem, ${\rm gcd}(2, \#G)\cdot   h^\ct_*(\O_{\ti \DD})=0$ in $\Cl(\Zl[G])$. Therefore, since $\l-1$ is even, by applying (\ref{723}) 
(for $d=0$) to the cover $X=\ti \DD\to Y= \DD$, we obtain
that 
$$
(\l-1)\cdot \left[h^\wedge_*( \theta^{\l}(T^\vee_{\ti\DD/\DD})^{-1})+h^\wedge_*(\pi^*(c')\cdot \theta^{\l}( T^\vee_{\ti\DD/\DD})^{-1})\right]=0
$$
in $\Cl(\Zl[G])^\wedge$ where $c'\in \Kr_0(\DD )'=\Gr_0(\DD )'$ is supported on a proper closed subset of $\DD$.
We can assume that $c'$ is supported away from primes in $S$. Then, as in (a) above 
we obtain
that the second term in the above sum vanishes. Therefore,  
$$
(\l-1)\cdot (h^\wedge_*( \theta^{\l}(T^\vee_{ \ti \DD/\DD})^{-1}))=0
$$
also.  
Observe that $ q:\DD\to D$ is an isomorphism over an open subset of $\DD$
whose complement has image in $X$ disjoint from the support of $T^\vee_{X/Y}$.
Also the formation
of  cotangent complexes of  flat morphisms commutes with base change (\cite{IllusieCotangent}), 
we can   see that
 the base change of $\theta^{\l}(T^\vee_{X/Y})^{-1} \in \Kr_0(G, X )^\wedge$ to $ \ti \DD$
is equal to $\theta^{\l}(T^\vee_{ \ti \DD /\DD })^{-1}\in \Kr_0(G, \ti \DD)^\wedge$.
Using these two facts, and the projection formula, we now obtain
\begin{eqnarray*}
(\l-1)\cdot \cl_{X_{\O_N}/Y_{\O_N}}(\iota_*[ \O_{D}]) 
=\ \ \ \ \ \ \ \ \ \ \ \ \ \ \ \ \ \ \ \ \ \ \ \ \ \ \ \ \ \ \ \ \ \\
=(\l-1)\cdot f^\wedge_*(\pi^*(\iota_*[\O_{D}])\cdot \theta^{\l}(T^\vee_{X_{\O_N}/Y_{\O_N}})^{-1})= \\
 =(\l-1)\cdot f^\wedge_*(\pi^*(\iota_*[ q_* \O_{\DD}])\cdot \theta^{\l}(T^\vee_{X_{\O_N}/Y_{\O_N}})^{-1})=\\
=(\l-1)\cdot h^\wedge_*( \theta^{\l}(T^\vee_{ \ti \DD /\DD })^{-1})=0.
\end{eqnarray*}
Here, for simplicity, we also write $\pi$ for the cover $ X_{\O_N}\to Y_{\O_N}$
and  denote by $f$  the structure morphism $X_{\O_N}\to \Spec(\Z)$.
As above we can now see that we have
$$
\cl_{X/Y}(\delta)=\sum\nolimits_i m_i\cdot \cl_{X_{\O_N}/Y_{\O_N}}((\iota_i)_*[\O_{D_i}])
$$
 and this, together with the above, concludes the proof of part (2).
\end{proof}

\begin{cor}\label{corvanishF1}
Under the above assumptions, if $\GG$ is a locally free coherent $\O_Y$-module of rank $r$,
then 
$$
{\rm gcd}(2, \#G)\cdot ( \overline\chi(X, \pi^*\GG)-r\cdot \overline\chi(X, \O_X))=0
$$
in $\Cl(\Z[G])$.
\end{cor}

\begin{proof}
This  follows from  Proposition \ref{vanishF1} (1), since $[\GG]-r\cdot [\O_Y]\in F_1\Gr_0(Y)$.
\end{proof}
\end{para}

\begin{subsection}{The input localization theorem}
 
Here we let $S$  be the smallest finite set of rational primes which
contains the support of the branch locus of $\pi: X\to Y$.
For simplicity, let us set   $X_S=\cup_{p\in S} X_{{\mathbb F}_p}$,
  $\widehat X_S=\cup_{p\in S} X_{\Z_p}$.

\begin{thm}\label{thminput}
Let $\pi: X\to Y$ be a tamely ramified $G$-cover of schemes which are projective and
flat over $\Spec(\Z)$ of relative dimension $1$. Suppose that $Y$ is regular and that $X$ is a local complete
intersection. Let $\F$ be a $G$-equivariant coherent   $\O_X$-module. Then 
$$
{\rm gcd}(2, \#G)^{v_2(\#G)+2}{\rm gcd}(3, \#G)^{v_3(\#G)-1}\cdot
\overline\chi(X, \F)
$$
in $\Cl(\Z[G])$ depends only on  the pair $(\widehat X_S, \F|_{\widehat X_S})$
where $\F|_{\widehat X_S}$ denotes
the pull-back of $\F$ from $X$ to $\widehat X_S$. \end{thm}

 \begin{proof} 
 We consider the projections
 $\overline\chi(X, \F)_\rho$  on the localizations 
 $ \Cl(\Z[G])_\rho$ of 
 the finite $\Gr_0(\Z[G])$-module $\Cl(\Z[G])$ at the maximal ideals $\rho\subset \Gr_0(\Z[G])$.
 Recall
 $$
  \Cl(\Z[G])=\oplus_\rho \Cl(\Z[G])_\rho.
 $$
 Consider $\rho$ that do not contain the kernel $I_G$ 
 of the rank map. The projection 
 $\overline\chi(X, \F)_\rho$ depends only on the inverse image $(\iota_*)_\rho^{-1}([\F])$ of the class
 of $\F$ under the isomorphism (\cite[Theorem 6.1]{CEPTRR})
 \begin{equation}
(\iota_*)_\rho:  \Gr_0(G, X^\rho)_\rho\xrightarrow{\sim} \Gr_0(G, X)_\rho
 \end{equation}
 where $\iota: X^\rho\subset X$. For such $\rho$, the fixed point subscheme $X^\rho $ is contained in the
 ramification locus $R=R(X/Y)\subset X_S$ and there is a similar isomorphism 
  \begin{equation}
(\widehat\iota_*)_\rho:  \Gr_0(G, X^\rho)_\rho\xrightarrow{\sim} \Gr_0(G, \widehat X_S)_\rho
 \end{equation}
 such that  $(\iota_*)_\rho$ and $(\widehat\iota_*)_\rho$  commute with   the base change
 homomorphism  $\Gr_0(G, X)\to \Gr_0(G, \widehat X_S)$,
 $\F\mapsto \F|_{\widehat X_S}$. This shows that $(\iota_*)_\rho^{-1}([\F])$
 and therefore also $\overline\chi(X, \F)_\rho$, for $I_G\not\subset\rho $, 
 only depends on   $(\widehat X_S, \F | _{\widehat X_S})$. It remains to deal with maximal $\rho$ 
 such that $I_G\subset \rho$. These are of the form $\rho=\rho_{(1, p)}$ for some prime $p$ that divides $\# G$. The argument in the proof of \cite[Proposition 4.5]{PaCubeInvent}
 shows that for such $\rho$, the component $\overline\chi(X, \F)_\rho$
 depends only on the $p$-power part $\overline\chi(X, \F)_p$ of the restriction of $ \overline\chi(X, \F)$ to the $p$-Sylow $G_p$. 
 To avoid a conflict in the notation we will use   in this proof the symbol $q$ to denote a prime in the set $S$.
   
  Note that, under our assumptions, $\pi: X\to Y$ is finite and flat. If $\GG=(\pi_*(\F))^G$, there is a 
 canonical short exact sequence of $G$-equivariant coherent $\O_X$-modules  
 $$
 0\to \pi^*\GG\to \F\to {\mathscr Y}\to 0,
 $$
 with ${\mathscr Y}$ supported on the ramification  locus $R(X/Y)$. 
 In fact, $\mathscr Y$ is canonically isomorphic to the cokernel of $\pi^*\GG |_{\widehat X_S}\to \F|_{\widehat X_S}$
 and $\GG |_{\widehat X_S}=(\pi_*(\F |_{\widehat X_S}))^G$. This shows that 
 ${\mathscr Y}$ is determined from $\F |_{\widehat X_S}$. Therefore, it is enough to show
 the statement for sheaves of the form $\F=\pi^*\GG$. In view of Corollary \ref{corvanishF1},
we first consider 
   the case $\F=\O_X$. Notice that $\pi_*\O_X$ is $\O_Y$-locally free
 of rank $\#G$ on $Y$ and hence, again by  Corollary \ref{corvanishF1}, the difference 
 $$
 \overline\chi(X, \pi^*(\pi_*(\O_X)))-\#G\cdot \overline\chi(X, \O_X)
 $$ is ${\rm gcd}(2,\#G)$-torsion. On the other hand, the $G$-action morphism, $m: X\times G\to X\times_Y X$, $(x, g)\mapsto (x, x\cdot g)$,
restricts to an isomorphism over   $U$. This gives a 
 $G$-equivariant  homomorphism
 $$
 \pi^*(\pi_*(\O_X))=\pi_*(\O_X)\otimes_{\O_Y} \O_X\xrightarrow{m^*}\bigoplus_{g\in G}\O_X={\rm Maps}(G,\O_X) 
 $$
 of $G$-equivariant coherent $\O_X$-modules which is injective.  The cokernel of $m^*$ is supported on $R(X/Y)$, and, as above,
is determined from $\widehat X_S$. Since  we have
 $\overline\chi(X, {\rm Maps}(G,\O_X) )=0$,
 this shows that ${\rm gcd}(2,\#G)\cdot (\#G)\cdot \overline\chi(X, \O_X)$ depends only on $\widehat X_S$. Therefore, ${\rm gcd}(2,\#G)^{v_2(\#G)+1}\cdot \overline\chi(X, \O_X)_2$ depends only on $\widehat X_S$. By Corollary \ref{corvanishF1} and the above, 
we see that
${\rm gcd}(2, \#G)^{v_2(\#G)+2} \cdot
\overline\chi(X, \F)_2$ depends only on $(\widehat X_S, \F|_{\widehat X_S})$.
 
 We now deal with the $p$-power part $\overline\chi(X, \O_X)_p$ for $p$ odd.
By the above, it is enough to consider the case that $G$ is a $p$-group.
We claim that, in this case,  the $I_G$-adic completion
$\Cl(\Zl[G])^\wedge$ is the $p$-power part $\Cl(\Zl[G])_p$: Indeed, we observe that the class  $[\Zl[G]]\in \Gr_0(\Z[G])$ annihilates $\Cl(\Zl[G])$ but
since it has rank $\#G$, it is invertible in the localizations 
of $\Gr_0(\Z[G])$ at  $\rho=I_G+(q)$, for all $q\neq p$.
Hence, the completion $\Cl(\Zl[G])^\wedge$ is supported at $p$ and the claim follows
since, by \ref{idealspara}, the only  prime ideal of $\Gr_0(\Z[G])$ supported 
over $p$ is $I_G+(p)$.
(cf. the proof of \cite[Proposition 4.5]{PaCubeInvent}). 
Combining   Proposition \ref{vanishF1} (2) 
and (\ref{723}) we obtain:
\begin{equation}\label{finalARR}
(\l-1) (\psi^\l-\l^{-1})\cdot \overline\chi(X, \O_X)=\l^{-1}(\l-1)   \cdot f^\wedge_*(r_{X/Y}^{\l})\ \ \ \ 
\end{equation}
in $\Cl(\Zl[G])^\wedge$. Now apply (\ref{finalARR})  to a prime $\l$ as in Proposition \ref{cft}. We  see that 
 Proposition \ref{bernoulli} implies that the multiple ${\rm gcd}(3, \#G)^{v_3(\#G)-1}\cdot \overline\chi(X, \O_X)_p$ is determined by 
 $f^\wedge_*(r^{\l}_{X/Y})\in \Cl(\Zl[G])^\wedge$. 
 
 We will show that $f^\wedge_*(r^{\l}_{X/Y})$ depends only on the $G$-cover $\widehat X_S\to \widehat Y_S$.
 Set $U_S=X-X_S$.
 Recall 
 $$
 r^{\l}_{X/Y}=\theta^{\l}(T^\vee_X)^{-1}\cdot \theta^{\l}(\pi^* T^\vee_Y)-1
 $$
is in $\Kr_0(G, X)^\wedge$ with trivial image in $\Kr_0(G, U_S)^\wedge$ under restriction. 
Observe that $f^\wedge_*(r^{\l}_{X/Y})$ only depends on the image of  $r^{\l}_{X/Y}$
in $\Gr_0(G, X)^\wedge$.
We have a commutative diagram
 \begin{equation*}\label{locdiagram}
 \begin{matrix}
 \Gr_1(G, U_S)^\wedge\to &\Gr_0(G, X_S)^\wedge\to & \Gr_0(G, X)^\wedge\to &\Gr_0(G, U_S)^\wedge\to 0\\
 \downarrow  &\downarrow\ \ \  &\beta\downarrow\ \ \ &\downarrow\ \ \ \ \ \ \\
  \Gr_1(G, \widehat U_S)^\wedge\to &\Gr_0(G, X_S)^\wedge\to  &\Gr_0(G, \widehat X_S)^\wedge\to &\Gr_0(G, \widehat U_S)^\wedge\to 0,\\
 \end{matrix}
 \end{equation*}
 where the rows are exact and are obtained by completing the standard equivariant localization sequences (\emph{cf.} \cite{ThomasonEquiv}).
 Here $\widehat U_S=\cup_{q\in S} X_{\Q_q}$,
 the vertical maps are given by base change;
the second vertical map is the identity.
Since the action of $G$ on $\widehat U_S$ is free, we have by \'etale descent
$  \Gr_1(G, \widehat U_S)=\Gr_1(\widehat U_S/G)=\oplus_{q\in S} \Gr_1(Y_{\Q_q})$.
(In particular, this also implies that $I^m_G\cdot \Gr_1(G, \widehat U_S)=(0)$ 
for $m>1$ and so  $\Gr_1(G, \widehat U_S)^\wedge= \Gr_1(G, \widehat U_S)$.)
As in \cite{ChinburgTameAnnals}, \cite{CEPTAnnals}, we can see that 
$f^\wedge_*\cdot (i_S)^\wedge_*: \Gr_0(G, X_S)^\wedge\to \Gr_0(G, X)^\wedge\to  \Cl(\Zl[G])^\wedge$
can be written as a composition $(f_S)_*: \Gr_0(G, X_S)^\wedge\to \oplus_{q\in S}\Kr_0({\mathbb F}_q[G])^\wedge\to \Cl(\Zl[G])^\wedge$ where the first arrow is given by the sum of 
the projective equivariant Euler characteristics for the $G$-schemes $f_q: X_{{\mathbb F}_q}\to \Spec({\mathbb F}_q)$. 

\begin{prop}\label{vanProp}
The
homomorphism $(f_S)_*$ vanishes on the image of the map $\Gr_1(G, \widehat U_S)^\wedge\to\Gr_0(G, X_S)^\wedge$.
\end{prop}
\begin{proof}
For this we need:
\begin{lemma}
The maps above give a  commutative diagram
 \begin{equation*}
 \begin{matrix}
 \Gr_1(Y_{\Q_q})&\xrightarrow{\sim} & \Gr_1(G, X_{\Q_q}) &\to &\Gr_0(G, X_{{\mathbb F}_q}) \\
  (g_{\Q_q})_*\downarrow\ \ \ \  & & \ \ \ \ \ &&(f_q)_* \downarrow \ \ \ \ \ \\
  \ \ \ \ \Q_q^*  \ \ & \xrightarrow {} &\Gr_1(\Q_q[G]) &\to &  \Gr_0({\mathbb F}_q[G]) \\
  \end{matrix}
 \end{equation*}
 where the bottom left horizontal arrow 
 $$
  \Q_q^*\to  \Gr_1(\Q_q[G])={\rm Hom}_{{\rm Gal}(\bar\Q_q/\Q_q)}(\scrR_{\bar\Q_q}(G), \bar\Q_q^*)
  $$ sends $\mu\in \Q_q^*$ to the character function $\chi\mapsto \mu^{{\rm deg}(\chi)}$.
  \end{lemma}
  \begin{proof}  Using the Quillen-Gersten spectral sequence we can see that 
  $ \Gr_1(Y_{\Q_q})$ is generated by the following two types of elements: 
  Constants $c\in \Q_q^*$ considered as giving automorphisms of the structure sheaf
  of $Y_{\Q_q}$ (\emph{i.e.} elements in the image of $(f_{\Q_q})^*: \Q_q^*=\Gr_1(\Spec(\Q_q))^*\to 
  \Gr_1(Y_{\Q_q})$), and elements in the image of $k(y)^*=\Gr_1(\Spec(k(y)))\to \Gr_1(Y_{\Q_q})$
  where $k(y)$ is the residue field of some closed point $y$ of $Y_{\Q_q}$ (see \cite[Example 4.6]{GilletAdv}).
  It is enough to check the commutativity on these two types of elements: 
  The statement for the first type follows easily from the fact 
  (a consequence of Lefschetz-Riemann-Roch, or of the main result of \cite{NakajimaInv}) that the $G$-character
  $[{\rm H}^0(X_{\Q_q}, \O_{X_{\Q_q}})]-[{\rm H}^1(X_{\Q_q}, \O_{X_{\Q_q}})]$ is a multiple of
  the regular character. The statement for the second type follows from an
   explicit calculation by using the normal basis theorem.
   \end{proof}
 
 The Lemma implies that the values of $(f_S)_*$ on the image of $\Gr_1(G, \widehat U_S)^\wedge\to\Gr_0(G, X_S)^\wedge$ are in the subgroup generated by the sums of images of
 the character functions $\chi\mapsto \mu^{{\rm deg}(\chi)}$.
These values are all
classes of virtually free ${\mathbb F}_q[G]$-modules in $\Kr_0({\mathbb F}_q[G])\subset \Gr_0({\mathbb F}_q[G])$
and this shows the desired vanishing.
\end{proof}
\smallskip

Proposition \ref{vanProp}, combined with a chase on the commutative
diagram of the previous page, now implies that $f^\wedge_*(r^\l_{X/Y})$  only depends on 
 the image of $r^\l_{X/Y}$ under the base change 
 $
 \beta: \Gr_0(G, X)^\wedge\to \Gr_0(G, \widehat X_S)^\wedge$.
 By flat base change for the cotangent complex \cite{IllusieCotangent} this image is the class of
 $$
 r^\l_{\widehat X_S/\widehat Y_S}:=\theta^\l(T_{\widehat X_S}^\vee)^{-1}\cdot \theta^\l(\pi^*T_{\widehat Y_S}^\vee)-1.
 $$
 Therefore, $f^\wedge_*(r^\l_{X/Y})$ only depends on $\widehat X_S\to \widehat Y_S$,
which, in turn, since $\widehat Y_S=\widehat X_S/G$, only depends on the $G$-scheme
$\widehat X_S$.
 
Combining all these, we now obtain the statement of Theorem \ref{thminput}.
\end{proof}
\end{subsection}


 \bibliographystyle{plain}
 
\bibliography{Galois}

\end{document}